\documentclass[reqno,11pt,psamsfonts]{amsart}

\usepackage[all]{xy}
\usepackage{amsthm}
\usepackage{amssymb}
\usepackage{amsmath,amscd}
\usepackage{mathrsfs}
\usepackage{color}
\usepackage{enumerate}
\usepackage{multicol}
\usepackage{longtable}
 \usepackage{geometry}
\usepackage[dvipdfmx]{graphicx}
\usepackage[all]{xy}
\usepackage{tikz}
\usepackage{qcircuit}

\def\a{\alpha}
\def\b{\beta}

\def\d{\delta}

\def\g{\gamma}
\def\l{\lambda}
\def\om{\omega}

\def\NN{{\mathbb N}}
\def\PP{{\mathbb P}}

\def\RR{{\mathbb R}}

\def\ZZ{{\mathbb Z}}

\def\cal{\mathcal}

\def\cA{{\cal A}}

\def\cD{{\cal D}}

\def\cP{{\cal P}}

\def\bg{{\bf g}}

\def\bx{{\bf x}}

\def\bM{{\bf M}}

\def\Aut{\operatorname{Aut}}

\def\det{\operatorname{det}}
\def\dim{\operatorname{dim}}

\def\Ext{\operatorname {Ext}}

\def\GL{\operatorname {GL}}

\def\GKdim{\operatorname{GKdim}}
\def\GL{\operatorname{GL}}
\def\gldim{\operatorname{gldim}}

\def\GrMod{\operatorname{GrMod}}

\def\GrAut{\operatorname{GrAut}}

\def\id{\operatorname {id}}

\def\mod{\operatorname{mod}}

\def\PGL{\operatorname{PGL}}

\def\<{\langle}
\def\>{\rangle}

\def\klang{k \langle x_1,\dots,x_n \rangle}
\def\kxy{k \langle x,y \rangle}
\def\kxyz{k \langle x,y,z \rangle}

\def\xx{x_1 \otimes x_2}
\def\XY{x_1 \otimes y_2}
\def\yx{y_1 \otimes x_2}
\def\yy{y_1 \otimes y_2}

\def\rnum#1{\expandafter{\romannumeral #1}}
\def\Rnum#1{\uppercase\expandafter{\romannumeral #1}}

\theoremstyle{plain} 
\newtheorem{theorem}{Theorem}[section]

\newtheorem{lemma}[theorem]{Lemma}
\newtheorem{proposition}[theorem]{Proposition}
\newtheorem*{Mthm}{Main~Theorem}

\theoremstyle{definition}
\newtheorem{definition}[theorem]{Definition}

\theoremstyle{remark}

\newtheorem{remark}[theorem]{Remark}

\numberwithin{equation}{section}

\begin{document}
\pagenumbering{arabic}
	
\title[Classifications of $3$-dim. cubic AS-reg. alg. whose point schemes are not integral]
{Classifications of $3$-dimensional cubic AS-regular algebras whose point schemes are not integral}
	
\author{Ayako Itaba, Masaki Matsuno and Yu Saito}

\address{Institute of Arts and Sciences,
Tokyo University of Science, 6-3-1 Niijuku, Katsushika-ku, Tokyo 125-8585, JAPAN}
\email{itaba@rs.tus.ac.jp}
\address{School of General and Management Studies, 
Suwa University of Science 5000-1, Toyohira, Chino, Nagano 391-0292, JAPAN}
\email{matsuno\_masaki@rs.sus.ac.jp}
\address{Graduate School of Science and Technology, Shizuoka University, 836 Ohya, Suruga-ku, Shizuoka-shi, Shizuoka 422-8529, JAPAN}
\email{saito.yu.18@shizuoka.ac.jp}

	
	
	
\maketitle

\begin{abstract}
By the result of Artin--Tate--Van den Bergh, 
every $3$-dimensional cubic AS-regular algebra A can be expressed as a geometric algebra $A=\mathcal{A}(E,\sigma)$, 
where $E$ is either $\mathbb{P}^{1}\times \mathbb{P}^{1}$ or a curve of bidegree ($2$,$2$) in $\mathbb{P}^{1}\times \mathbb{P}^{1}$ and $\sigma\in \mathrm{Aut}_{k}E$. 
In particular, we treat the following three configurations: 
(1) a conic and two lines in a triangle, 
(2) a conic and two lines intersecting in one point, and 
(3) a quadrangle. 
For each of these cases, we (i) list all defining relations of the corresponding algebras $\mathcal{A}(E,\sigma)$, and (ii) classify them up to graded algebra isomorphism and graded Morita equivalence. 
Furthermore, we  present explicit (twisted) superpotentials whose derivation-quotient algebras realize these algebras
and verify that the resulting algebras are AS-regular. 
Combining our results with existing classifications for the remaining types (including Types P, S, T, WL, and TWL), 
we thereby complete the classification of 3-dimensional cubic AS-regular algebras whose point schemes are not integral. 
\end{abstract}

	

	

\section{Introduction}
Throughout this paper, we fix an algebraically closed field $k$ of characteristic $0$, 
and the $(n-1)$-dimensional projective space over $k$ is denoted by $\mathbb{P}_{k}^{n-1}=\mathbb{P}^{n-1}$. 
The definition  of an Artin-Schelter regular (shortly, AS-regular) algebra was introduced by Artin--Schelter \cite{AS}. 
Note that an AS-regular algebra is a noncommutative analogue of a commutative polynomial ring.  
Also, Artin--Schelter \cite[Theorem 1.5]{AS} proved that every $3$-dimensional AS-regular algebra finitely generated in degree $1$ over $k$ is isomorphic to
one of the following forms:
	$k \langle x,y,z \rangle/(f_1,f_2,f_3)$
	where $f_i$ are homogeneous elements of degree $2$ (the {\it quadratic} case), or
	$k \langle x,y \rangle/(g_1,g_2)$
where $g_j$ are homogeneous elements of degree $3$ (the {\it cubic} case).
Artin-Tate-Van den Bergh \cite{ATV1} proved that every $3$-dimensional AS-regular algebra finitely generated in degree $1$
determines and is determined by the pair $(E,\sigma)$ where $E$ is a scheme and $\sigma$ is an automorphism of $E$.
Moreover, if $A$ is the quadratic case, then $E$ is either $\mathbb{P}^2$ or a cubic curve in $\mathbb{P}^2$, and if it is the cubic case, then $E$ is either $\mathbb{P}^1 \times \mathbb{P}^1$ or a curve of bidegree $(2,2)$ in $\mathbb{P}^1 \times \mathbb{P}^1$.

In Noncommutative algebraic geometry, 
classifications of AS-regular algebras are one of the most important projects. 
Mori--Smith \cite{MoS2} gave the classification of all superpotentials whose Jacobian algebras are 3-dimensional noetherian quadratic Calab--Yau algebras. 
Mori--Ueyama \cite{MU2} gave the classification of all superpotentials whose Jacobian algebras are 3-dimensional noetherian cubic Calab--Yau algebras. 
Mori \cite{M} introduced {\it a geometric algebra} for a quadratic algebra. 
Note that every $3$-dimensional quadratic AS-regular algebra $A$ is a geometric algebra for a quadratic algebra. 
From the point of view of a geometric algebra for a quadratic algebra, 
the first and second authors gave the complete list of defining relations of all $3$-dimensional quadratic AS-regular algebras and classify them up to graded algebra isomorphism and graded Morita equivalence (see \cite{IM1, IM2, Ma}). 
Recently, the second and third authors \cite{MaS} defined 
{\it a geometric algebra} for a cubic algebra 
by extending the notion of a geometric algebra for a quadratic algebra 
(see Definition \ref{def.GA} (\cite[Definition 3.3]{MaS})). 
Also, they gave the list of defining relations of some classes of $3$-dimensional cubic AS-regular algebras (called Type P, S, T) 
and classified them up to graded algebra isomorphism and graded Morita equivalence
(see Remark \ref{Rem-MaS}). 
As a continuation of these studies, 
we focus on $3$-dimensional {\it cubic} AS-regular algebras in this paper. 

Let $A=k \langle x_1, \dots, x_n \rangle/(f_1, \dots, f_m)$ be a cubic algebra
where $\deg x_i=1$ $(i=1,\dots,n)$ and $f_j$ is a homogeneous element of degree $3$ $(j=1, \dots, m)$.
Note that every $3$-dimensional cubic AS-regular algebra $A$ is a geometric algebra for a cubic algebra. 
Moreover, the point scheme $E$ of $A$ is $\mathbb{P}^1 \times \mathbb{P}^1$ or a curve of bidegree $(2,2)$ in $\mathbb{P}^1 \times \mathbb{P}^1$. 
Our main result is as follows: 
\begin{Mthm}[{Theorems \ref{thm.isom}, \ref{thm.grmod}}]
\label{thm}
Let $A=\mathcal{A}(E,\sigma)$ be a $3$-dimensional cubic AS-regular algebra. 
Assume that $E$ is either {\rm (1)} a conic and two lines in a triangle,
{\rm (2)} a conic and two lines intersecting in one point, or {\rm (3)} quadrangle. 
For each case, we give the list of defining relations of $A$ and
classify them up to graded algebra isomorphism and graded Morita equivalence in terms of their defining relations.
\end{Mthm}


Furthermore, we  present explicit (twisted) superpotentials whose derivation-quotient algebras realize these algebras and verify that the resulting algebras are AS-regular (Proposition \ref{prop.P}). 
Combining our results with existing classifications for the remaining types (including Types P, S, T, WL, and TWL), 
we thereby complete the classification of 3-dimensional cubic AS-regular algebras whose point schemes are not integral. 
For Types P, S and T, we refer the result 
by the second and third authors \cite{MaS} (see Remark \ref{Rem-MaS}). 
For Types  WL and TWL (when $E$ is not reduced), we refer the result of 
Artin--Tate--Van den Bergh \cite{ATV2} (see Subsection \ref{subsec.NR}). 

This paper is organized as follows: In Section \ref{sec.Pre}, we recall the definitions of an Artin--Schelter regular algebra from \cite{AS}, 
a twisted superpotential and its derivation-quotient algebra from \cite{BSW} and \cite{MoS1} (see Subsection \ref{subsec.AS}), 
and a geometric pair and a geometric algebra from \cite{MaS} (see Subsection \ref{subsec.GA}).
In particular, we describe the classification of $3$-dimensional cubic AS-regular algebras of Type WL and TWL (see Subsection \ref{subsec.NR}). 
In Section \ref{sec.Class}, we describe an approach to prove our results in this paper. 
At first, we study geometric pairs corresponding to $3$-dimensional cubic AS-regular algebras of Type S$'$, T$'$ and FL (Lemmas \ref{lem.step0} and \ref{lem.step1}). 
Next, we give a list of defining relations of them (Theorem \ref{thm.DR}). 
At the end of the section, we check AS-regularity of them (Proposition \ref{prop.P}).
In Section \ref{sec.Cond}, we classify them up to graded algebra isomorphism
(Theorem \ref{thm.isom})
and graded Morita equivalence (Theorem \ref{thm.grmod}). 
Finally, 
we summarize the complete classifications for the cases that point schemes are not integral as the tables in Subsection \ref{subsec.SUM}.

\section{Preliminaries}\label{sec.Pre}
Throughout this paper, all vector spaces and algebras are defined over $k$.
Assume that a graded algebra is an $\NN$-graded algebra $A=\bigoplus_{i \in \NN} A_i$. 
A graded algebra $A$ is called {\it connected} if $A_0=k$.
The category of graded right $A$-modules is denoted by $\GrMod A$. 
A morphisms in $\GrMod A$ is right $A$-module homomorphisms preserving a degree.
Graded algebras $A$ and $B$ are called {\it graded Morita equivalent} 
if the categories $\GrMod A$ and $\GrMod B$ are equivalent, 
denoted by $\GrMod A\cong \GrMod B$. 
\subsection{AS-regular algebras and twisted superpotentials. }\label{subsec.AS}
\begin{definition}[{\cite{AS}}]
	A connected graded algebra $A$ is called a {\it $d$-dimensional Artin--Schelter regular 
	{\rm (}
	AS-regular{\rm )} algebra} if $A$ satisfies 
	\begin{enumerate}[{\rm (i)}]
		\item $\gldim A=d < \infty$; 
		\item $\GKdim A:=\inf \displaystyle\left\{
	\a \in \RR \mid \dim_k \left(\sum_{i=0}^{n} A_i\right) \leq n^{\a} \,\,\,\textnormal{for all}\,\,\, n \gg 0
	\right\} < \infty$ ($\GKdim$ is called {\it the Gelfand--Kirillov dimension of $A$}); 
		\item $\Ext_{A}^{i}(k,A) \cong
		\begin{cases}
			k \quad \textnormal{if}\,\,\, i=d, \\
			0 \quad \textnormal{if}\,\,\, i \neq d, 
		\end{cases}$
		({\it Gorenstein conditions}). 
	\end{enumerate}
\end{definition}
Let $V$ be a finite dimensional vector space.
For an integer $m\geq 2$, the linear map $\varphi:V^{\otimes m} \to V^{\otimes m}$ is defined by
$\varphi(v_1 \otimes \dots \otimes v_{m-1} \otimes v_m):=v_m \otimes v_1 \otimes \dots \otimes v_{m-1}$.
The general linear group of $V$ is denoted by $\GL(V)$.
\begin{definition}[{\rm \cite[Introduction]{BSW}, \cite[Definition 2.5]{MoS1}}]
	Let $\om \in V^{\otimes m}$.
	\begin{enumerate}[{\rm (1)}]
		\item If $\varphi(\om)=\om$, then $\om$ is called a {\it superpotential}. 
		\item If there exists $\theta \in \GL(V)$ such that
		$(\theta \otimes \id^{\otimes m-1})(\varphi(\om))=\om$, then $\om$ is called a {\it twisted superpotential}. 
	\end{enumerate}
\end{definition}

For $n\geq 2$, let $V$ be an $n$-dimensional vector space.
Fix a basis $\{x_1,\dots,x_n\}$ for $V$.
For $\om \in V^{\otimes m}$, there exists a unique $\om_{i} \in V^{\otimes m-1}$ $(1\leq i \leq n)$ such that $\om=x_1 \otimes \om_1+\cdots+x_n \otimes \om_n$.
In this case, we define by $\partial_{x_i}\om:=\om_{i}$ 
the {\it left partial derivative} of $\om$ with respect to $x_i$ for $1 \le i \le n$. 
Similarly, there exists unique $\om'_i \in V^{\otimes m-1}$ 
$(1\leq i \leq n)$ such that 
$\om=\om'_1 \otimes x_1+\cdots+ \om'_n \otimes x_n$,
so we define by $\om \partial_{x_i}:=\om'_i$ the {\it right partial derivative} of $\om$ with respect to $x_i$ for $1 \le i \le n$. 

\begin{lemma}\label{lem.TSP}
    Let $V$ be a $2$-dimensional vector space with a basis $\{x_1, x_2\}$ and
    $\om \in V^{\otimes 4} \setminus \{0\}$. Then $\om$ is a twisted superpotential if
    and only if $(\partial_{x_1} \om, \partial_{x_2} \om)=(\om \partial_{x_1}, \om \partial_{x_2})$ as homogeneous two-sided ideals of $k \langle x_1,x_2 \rangle$.
\end{lemma}

\begin{proof}
	Let $\om \in V^{\otimes 4} \setminus \{0\}$.
	If $\om$ is a twisted superpotential, then there exists $\tau=
	\begin{pmatrix}
		a & b \\
		c & d
	\end{pmatrix}
	\in \GL_2(k)$ such that
	$(\tau \otimes \id^{\otimes 3})(\varphi(\om))=\om$. In this case, 
	$\partial_{x_1} \om=a \om\partial_{x_1}+c \om\partial_{x_2}$ and
	$\partial_{x_2} \om=b \om\partial_{x_1}+d \om\partial_{x_2}$. This means that
	$(\partial_{x_1} \om, \partial_{x_2} \om)=(\om \partial_{x_1}, \om \partial_{x_2})$ as homogeneous two-sided ideals of $k \langle x_1,x_2 \rangle$.
	
	Conversely, assume that $(\partial_{x_1} \om, \partial_{x_2} \om)=(\om \partial_{x_1}, \om \partial_{x_2})$ as homogeneous two-sided ideals of $k \langle x_1,x_2 \rangle$.
	If $\partial_{x_1} \om$ and $\partial_{x_2} \om$ are linearly independent, then
	there exists $\begin{pmatrix}
		a & b \\
		c & d
	\end{pmatrix} \in \GL_2(k)$ such that $\partial_{x_1} \om=a \om\partial_{x_1}+b \om\partial_{x_2}$ and
	$\partial_{x_2} \om=c \om\partial_{x_1}+d \om\partial_{x_2}$. In this case, we set
	$\tau :=\begin{pmatrix}
		a & c \\
		b & d
	\end{pmatrix} \in \GL_2(k)$. 
 Then 
	\begin{align*}
		(\tau \otimes \id^{\otimes 3})(\varphi(\om))
		& =  \tau(x_1) \otimes \om\partial_{x_1}+\tau(x_2) \otimes \om\partial_{x_2} \\
		& =(ax_1+cx_2) \otimes \om\partial_{x_1}+(bx_1+dx_2) \otimes \om\partial_{x_2} \\
		&=  x_1 \otimes (a\om\partial_{x_1}+b\om\partial_{x_2})+x_2 \otimes (c\om\partial_{x_1}+d\om\partial_{x_2}) \\
		& =x_1 \otimes \partial_{x_1} \om+x_2 \otimes \partial_{x_2}\om =\om.
	\end{align*}
	Assume that there exists $\a \in k$
	such that $\partial_{x_2} \om=\a \partial_{x_1} \om$. Since $(\partial_{x_1} \om, \partial_{x_2} \om)=(\om \partial_{x_1}, \om \partial_{x_2})$, there exist $\l,\mu \in k$
	such that $\om\partial_{x_1}=\l \partial_{x_1}\om$ and $\om\partial_{x_2}=\mu \partial_{x_1}\om$. Note that $(\l,\mu) \neq (0,0)$. 
	When $\a=0$ and $\l \neq 0$, then we set $\tau:=
	\begin{pmatrix}
		1/\l & -\mu/\l \\
		0 & 1 
	\end{pmatrix}$. 
	When $\a=0$ and $\l=0$, then we set 
	$\tau:=
	\begin{pmatrix}
		0 & 1 \\
		1/\mu & 0
	\end{pmatrix}$. 
	When $\a \neq 0$ and $\l=0$, then we set
	$\tau:=
	\begin{pmatrix}
		1 & 0 \\
		1/\mu & \a/\mu 
	\end{pmatrix}$. 
	When $\a \neq 0$ and $\mu=0$, then we set
	$\tau:=
	\begin{pmatrix}
	1/\l & \a/\l \\
	0 & 1 
	\end{pmatrix}$. 
	When $\a \neq 0$ and $\l\mu \neq 0$, then we set
	$\tau:=
	\begin{pmatrix}
	1/\l & 0 \\
	0 & \a/\mu 
	\end{pmatrix}$. 
	Therefore, for every case, 
	$\om$ is a twisted superpotential.
\end{proof}

The {\it derivation-quotient algebra} $\cD(\om)$ of $\om \in V^{\otimes m}$ is defined by
\[
\cD(\om):=\klang/(\partial_{x_1}\om, \dots, \partial_{x_n}\om). 
\]

\begin{remark}
By \cite[Lemma 2.2]{MoS2}, the linear span 
$\langle\partial_{x_1}\om, \dots, \partial_{x_n}\om \rangle_{k}$ 
does not depend on the choice of a basis for $V$.
\end{remark}
For a superpotential $\om\in V^{\otimes m}$ and $\theta\in \mathrm{GL}(V)$, 
$$
\om^{\theta}:=(\theta^{m-1}\otimes \theta^{m-2}\otimes \cdots \otimes \theta \otimes \mathrm{id})(\om)
$$
is called {\it the MS-twist of $\om$ by $\theta$} 
(\cite[page 390]{MoS1}, cf. \cite{IM2}). 
For $\om \in V^{\otimes 4}$, we set 
$
\Aut(\om):=\{\theta \in \GL_2(k) \mid (\theta^{\otimes 4})(\om)=\l \om \,\,\,\textnormal{ for some } \l \in k \setminus \{0\} \}$.
Note that $\Aut(\om)$ is a subgroup of $\GL_2(k)$.
For any element $\om \in V^{\otimes 4}$,
$\Aut(\om)$ becomes a subset of $\GrAut_k \cD(\om)$
(see \cite[Lemma 3.1]{MoS1}).

\begin{lemma}[{\rm \cite[Lemma 2.7]{MaS}}]\label{prop.MStwist}
Let $\om \in V^{\otimes 4}$ be a twisted superpotential and $\theta \in \Aut(\om)$.
Then the MS twist $\om^{\theta}$ of $\om$ by $\theta$ is a twisted superpotential.
\end{lemma}

\begin{theorem}[{\rm \cite[Proposition 2.9]{MoS1}}]\label{thm.TSP1}
\begin{enumerate}[{\rm (1)}]
\item For any $3$-dimensional quadratic AS-regular algebra $A$, 
there exists a unique twisted superpotential $\om$
which is a homogeneous polynomial of $\kxyz$ of degree $3$ up to non-zero scalar multiples such that $A=\cD(\om)$.
\item 
For any $3$-dimensional quadratic AS-regular algebra $A$, 
there exists a unique twisted superpotential $\om$ which is a homogeneous polynomial of $\kxy$ of degree $4$ up to non-zero scalar multiples such that $A=\cD(\om)$.
\end{enumerate}
\end{theorem}

By Theorem \ref{thm.TSP1}, the classification of $3$-dimensional AS-regular algebras finitely generated in degree $1$ 
can be reduced to the classification of twisted superpotentials whose derivation-quotient algebras are AS-regular.

Now, recall the condition to check whether a cubic algebra is AS-regular or not.
Let $V$ be a $2$-dimensional vector space with a basis $\{x_1,x_2\}$ 
and, $R$ a $2$-dimensional subspace of $V^{\otimes 3}$ with a basis $\{g_1,g_2\}$. 
The transpose of a matrix $N$ is denote by $N^{t}$ .
We write $\bx=
\begin{pmatrix}
	x_1 \\
	x_2
\end{pmatrix}$ and 
$\bg=
\begin{pmatrix}
	g_1 \\
	g_2
\end{pmatrix}$.
Then there exists a $2 \times 2$ matrix 
$\bM=\begin{pmatrix}
	m_{11} & m_{12} \\
	m_{21} & m_{22}
\end{pmatrix}$ 
whose entries belong to $V^{\otimes 2}$ such that $\bg=\bM\bx$.
A cubic algebra $T(V)/(R)$ is called {\it standard} 
if there exists a basis $\{x_1,x_2\}$ for $V$ and $\{g_1,g_2\}$ for $R$ such that
$(\bx^{t}\bM)^{t}=Q\bg$ for some $Q \in \GL_2(k)$.
When a cubic algebra $T(V)/(R)$ is standard, we regard entries of the matrix $\bM$
in the above as elements of the Segre product $k[x_1,x_2] \circ k[x_1,x_2]$.

\begin{theorem}[{\rm \cite[Theorem 1]{ATV1}}]
Let $V$ be a $2$-dimensional vector space and $R$ a $2$-dimensional subspace of $V^{\otimes 3}$. 
Then a cubic algebra $T(V)/(R)$ is a $3$-dimensional AS-regular algebra
if and only if $T(V)/(R)$ is standard and the common zero locus in $\PP^1 \times \PP^1$ of entries of the matrix $\bM$ in the above is empty.
\end{theorem}
%

\begin{lemma}[{\cite[Lemma 2.10]{MaS}}]
\label{prop.Stand}
Let $\om \in V^{\otimes 4}$ be a twisted superpotential. 
If $\partial_{x_1} \om$ and $\partial_{x_2} \om$ are linearly independent, 
then the derivation-quotient algebra $\cD(\om)$ is standard. 
\end{lemma}

%
\subsection{Geomertic algebras for cubic algebras. }\label{subsec.GA}
Let $V$ be a finite dimensional vector space and
$T(V)$ the tensor algebra on $V$. 
Suppose that $A$ is equal to a quotient algebra $T(V)/(R)$
of $T(V)$ where $R \subset V^{\otimes 3}$ is a subspace and
$(R)$ is the homogeneous two-sided ideal of $T(V)$ generated by $R$.
The dual space of $V$ is denoted by $V^{\ast}$.
Since $(V \otimes V \otimes V)^{\ast} \cong V^{\ast} \otimes V^{\ast} \otimes V^{\ast}$, every element $f \in R$ defines
a multilinear form from $V^{\ast} \times V^{\ast} \times V^{\ast}$ to $k$. 
For a cubic algebra $A=T(V)/(R)$, we define 
$\Gamma_A:=\{
(p_{1},p_{2},p_{3}) \in \PP(V^{\ast})^{\times 3} \mid f(p_{1},p_{2},p_{3})=0 \textnormal{ for all } f \in R
\}$. 
The $i$-th projection from $\PP(V^{\ast}) \times \PP(V^{\ast})$ to $\PP(V^{\ast})$ $(i=1,2)$ is denoted by $\pi_i$ . 
Two maps $\pi_{12}$ and $\pi_{23}$ from $\PP(V^{\ast})^{\times 3}$ to $\PP(V^{\ast}) \times \PP(V^{\ast})$ are defined 
as follows; for $(p_{1},p_{2},p_{3}) \in \PP(V^{\ast})^{\times 3}$,
	$\pi_{12}(p_{1},p_{2},p_{3}):=(p_{1},p_{2})$ and $\pi_{23}(p_{1},p_{2},p_{3}):=(p_{2},p_{3})$.
For a projective variety $E \subset \PP(V^{\ast}) \times \PP(V^{\ast})$, we define
	$$\Aut_k^G E:=\{
	\sigma \in \Aut_k E \mid (\pi_1 \circ \sigma)(p_{1},p_{2})=\pi_2(p_{1},p_{2}) \textnormal{ for all } (p_{1},p_{2}) \in E
	\}.
$$
A pair $(E,\sigma)$ is called {\it geometric} if 
$E \subset \PP(V^{\ast}) \times \PP(V^{\ast})$ is a projective variety and 
$\sigma \in \Aut_k^G E$.

\begin{definition}[{\rm \cite[Definition 3.3]{MaS}, cf. \cite[Definition 4.3]{M}}]
\label{def.GA}
Let $A=T(V)/(R)$ be a cubic algebra where $R$ is a subspace of $V^{\otimes 3}$.
\begin{enumerate}[{\rm (1)}]
\item We say that {\it $A$ satisfies {\rm (G$1$)}} if there exists a geometric pair $(E,\sigma)$ such that
$\Gamma_A=\{(p_{1},\,p_{2},\,(\pi_2 \circ \sigma)(p_{1},p_{2})) \in \PP(V^{\ast})^{\times 3} \mid (p_{1},\,p_{2}) \in E\}$. 
In this case, we write $\cP(A)=(E,\sigma)$. 
We call $\Gamma_{A}$ the {\it point scheme} of $A$. 
\item We say that {\it $A$ satisfies {\rm (G$2$)}} if there exists a geometric pair $(E,\sigma)$ such that 
$R=\{f \in V^{\otimes 3} \mid f(p_1,p_2,(\pi_2 \circ \sigma)(p_1,p_2))=0 \textnormal{ for all } (p_1,\,p_2) \in E\}$. 
In this case, we write $A=\cA(E,\sigma)$.
\item We say that $A$ is a {\it geometric algebra {\rm (}for a cubic algebra{\rm)}} if $A$ satisfies (G1) and (G2) with $A=\cA(\cP(A))$.\vspace{-5pt}
	\end{enumerate}
\end{definition}
\begin{lemma}[{\rm \cite[Theorems 3.5, 3.6]{MaS}}]\label{thm.GA}
Let $A=T(V)/(R)$ and $A'=T(V)/(R')$ be geometric algebras with $\cP(A)=(E,\sigma)$ and $\cP(A')=(E',\sigma')$ where
$E$ and $E'$ are projective varieties in $\PP(V^{\ast}) \times \PP(V^{\ast})$ and $\sigma \in \Aut_k^G E$, $\sigma' \in \Aut_k^G E'$.
Then the following statements hold{\rm :}
\begin{enumerate}[{\rm (1)}]
\item \begin{minipage}[t]{32em}
	       Two graded algebras $A$ and $A'$ are isomorphic  
		if and only if there exists an automorphism $\tau$ of $\PP(V^{\ast})$ such that $(\tau \times \tau)(E)=E'$ and the following diagram commutes{\rm :} \end{minipage}\hfill
		\begin{minipage}[t]{6em} 
		{\small
		$
		\xymatrix{
			E \ar[r]^{\tau \times \tau} \ar[d]_{\sigma} & E' \ar[d]^{\sigma'} \\
			E \ar[r]_{\tau \times \tau} & E'
		}
		$}
		\end{minipage}
\item \begin{minipage}[t]{32em}
		Two categories $\GrMod A$ and $\GrMod A'$ are equivalent 
		if and only if there exists a sequence $\{\tau_n\}_{n \in \ZZ}$ of automorphisms of $\PP(V^{\ast})$
		such that $(\tau_{n} \times \tau_{n+1})(E)=E'$ and the following diagram commutes for every $n \in \ZZ${\rm :}
		 \end{minipage}\hfill
		\begin{minipage}[t]{7em}
		{\small 
		$
		\xymatrix@C=40pt{
			E \ar@<0.5ex>[r]^{\tau_{n} \times \tau_{n+1}} \ar[d]_{\sigma} & E' \ar[d]^{\sigma'} \\
			E \ar@<-0.5ex>[r]_{\tau_{n+1} \times \tau_{n+2}} & E'
		}
		$}
		\end{minipage}
	\end{enumerate}
\end{lemma}

By Lemma \ref{thm.GA}, a classification of geometric algebras (for cubic algebra)
up to graded algebra isomorphism 
or graded Morita equivalence is reduced to the classification of geometric pairs.

\begin{definition}[{\rm \cite[Definition 3.7]{MaS}}]\label{def.equiv}
Let $V$ be a finite dimensional 
vector space and $E$ and $E'$ projective varieties in $\PP(V^{\ast}) \times \PP(V^{\ast})$. 
\begin{enumerate}[{\rm (1)}]
\item If there exist $\tau_1, \tau_2 \in \Aut_k \PP(V^{\ast})$ such that $E'=(\tau_1 \times \tau_2)(E)$, then we say that {\it $E$ and $E'$ are equivalent}, denoted by $E \sim E'$.
\item If there exists $\tau \in \Aut_k \PP(V^{\ast})$ such that $E'=(\tau \times \tau)(E)$, 
then we say that {\it $E$ and $E'$ are $2$-equivalent}, denoted by $E \sim_2 E'$. 
	\end{enumerate}
\end{definition}

Let $A=T(V)/(R)=\cA(E,\sigma)$ and $A'=T(V)/(R')=\cA(E',\sigma')$ 
be geometric algebras where
$E$ and $E'$ are projective varieties in $\PP(V^{\ast}) \times \PP(V^{\ast})$, and $\sigma \in \Aut_k^G E$, $\sigma' \in \Aut_k^G E'$.
It is clear that, if $E$ and $E'$ are $2$-equivalent, then they are equivalent.
Lemma \ref{thm.GA} shows that, if $A$ and $A'$ are graded algebra isomorphic (resp. graded Morita equivalent), then
$E$ and $E'$ are $2$-equivalent (resp. equivalent), so we need to classify projective varieties in $\PP(V^{\ast}) \times \PP(V^{\ast})$
up to $2$-equivalence (resp. equivalence) as a first step of the classification of geometric algebras up to graded algebra isomorphism 
(resp. graded Morita equivalence).

Artin--Tate--Van den Bergh \cite{ATV1} proved that, 
if $A$ is a $3$-dimensional cubic AS-regular algebra, then the point scheme $\Gamma_A$ of $A$ is isomorphic to
either $\PP^1 \times \PP^1$ or a curve of bidegree $(2,2)$ in $\PP^1 \times \PP^1$.
More precisely, every $3$-dimensional cubic AS-regular algebra $A$ determines the pair $(E,\sigma)$, and $A$ is determined by
the pair $(E,\sigma)$, 
where $E$ is either $\PP^1 \times \PP^1$ or a curve of bidegree $(2,2)$ in $\PP^1 \times \PP^1$ and
$\sigma$ is an automorphism of $E$ satisfying $\pi_1 \circ \sigma=\pi_2$. 


As a first step of giving the classification of $3$-dimensional cubic AS-regular algebras 
in terms of geometric algebras, 
we need to study curves of bidegree $(2,2)$ in $\PP^1 \times \PP^1$.
Note that Belmans \cite{B} classified curves of bidegree $(2,2)$ in $\PP^1 \times \PP^1$ up to isomorphism (see \cite[Table 3]{B} for details):
\begin{multicols}{2}
	\begin{itemize}
		\item elliptic curve 
		\item cuspidal curve 
		\item nodal curve 
		\item two conics in general position 
		\item two tangent conics 
		\item a conic and two lines in a triangle 
		\item a conic and two lines intersecting in one point 
		\item quadrangle 
		\item twisted cubic and a bisecant 
		\item twisted cubic and a tangent line 
		\item double conic 
		\item two double lines 
		\item a double line and two lines 
	\end{itemize}
\end{multicols}


\begin{lemma}\label{lem.subvar}
Let $E \subset \PP^1 \times \PP^1$ be a projective variety.
If $\Aut_k^G E \neq \emptyset$, 
then the following conditions are equivalent{\rm :}
\begin{enumerate}[{\rm (1)}]
   \item 
    A projective variety $\ell$ of bidegree $(0,1)$ is a subvariety of $E$. 
    \item A projective variety $\ell'$ of bidegree $(1,0)$  is a subvariety of $E$. 
    \end{enumerate}
\end{lemma}

\begin{proof}
Let $\ell=\PP^1 \times \{P\}$ be a projective variety of bidegree $(0,1)$ such that $\ell \subset E$.
Let $\sigma \in \Aut_{k}^G E$. Then $\sigma(\ell) \subset \{P\} \times \PP^1$.
Since $\{P\} \times \PP^1 \cong \PP^1$, if $\sigma(\ell) \neq \{P\} \times \PP^1$, then
$\sigma(\ell)$ is isomorphic to a proper subvariety of $\PP^1$.
This means that $\sigma(\ell)$ is a finite set. 
But, since $\sigma$ is injective, $\sigma(\ell)$ is not a finite set,
so this contradicts.

Conversly, 
let $\ell'=\{P\} \times \PP^1$ be a projective variety of bidegree $(1,0)$ such that $\ell' \subset E$.
Let $\sigma \in \Aut_{k}^G E$. Then $\sigma^{-1}(\ell') \subset \PP^1 \times \{P\}$.
Since $\PP^1 \times \{P\} \cong \PP^1$, if $\sigma^{-1}(\ell') \neq \PP^1 \times \{P\}$, then
$\sigma^{-1}(\ell')$ is isomorphic to a proper subvariety of $\PP^1$.
This means that $\sigma^{-1}(\ell')$ is a finite set.
However, since $\sigma^{-1}$ is injective, $\sigma^{-1}(\ell')$ is not a finite set,
so this contradicts.
\end{proof}

\begin{remark}
Lemma \ref{lem.subvar} tells us that 
every curve of bidegree $(2,2)$ in $\PP^1 \times \PP^1$ does not appear
as the point scheme of a $3$-dimensional cubic AS-regular algebra; 
suppose that $E$ is one of the following in  $\PP^1 \times \PP^1$: 
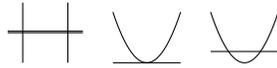
\begin{center}
$
\begin{tikzpicture}[scale=0.4]
\draw (1.5,-1)--(1.5,1);
\draw (0,-1)--(0,1);
\draw (-0.5,0)--(2,0);
\draw (-0.5,0.05)--(2,0.05);
\end{tikzpicture}
\quad
\begin{tikzpicture}[scale=0.3]
\draw[samples=150, domain=-1.5:1.5]plot(\x,{(\x * \x)});
\draw (-1.5,0)--(1.5,0);
\end{tikzpicture}
\quad
\begin{tikzpicture}[scale=0.3]
\draw[samples=150, domain=-1.5:1.5]plot(\x,{(\x * \x) - 0.5});
\draw (-1.5,0)--(1.5,0);
\end{tikzpicture}
$
\end{center}
By Lemma \ref{lem.subvar},  $\Aut_k^G E = \emptyset$. 
Considering defining relations of 3-dimensional cubic 
AS-regular algebras from the view of a geometric algebra, 
we exclude the cases when $E$ is one of the above figures.
\end{remark}
The aim of this paper is to give the complete list of defining relations 
of $3$-dimensional cubic AS-regular algebras whose point schemes are not integral.
In this paper, we define the types of the point scheme $E$ of $3$-dimensional cubic AS-regular algebras as follows:
\begin{center}
\begin{longtable}{|c|l|c|}
\hline Type & E & Figures\\ \hline\hline
{\rm Type P} & $E$ is $\PP^1 \times \PP^1$ & ------ \\ \hline
{\rm Type S} & $E$ consists of two conics in general position. 
&
\begin{tikzpicture}[scale=0.2, rotate=135]
\draw[samples=150, domain=-0.5:1]plot(\x,{sqrt(\x+0.5)});
\draw[samples=150, domain=-0.5:1]plot(\x,-{sqrt(\x+0.5)});
\draw[samples=150, domain=-1:0.5]plot(\x,{sqrt(0.5 - \x) -0.05});
\draw[samples=150, domain=-1:0.5]plot(\x,-{sqrt(0.5 - \x) +0.05});
\end{tikzpicture}
\\ \hline
{\rm Type T} &  $E$ consists of two tangent conics.
&
\begin{tikzpicture}[scale=0.2, rotate=135]
\draw[samples=150, domain=0:1.5]plot(\x,{sqrt(\x)});
\draw[samples=150, domain=0:1.5]plot(\x,-{sqrt(\x)});
\draw[samples=150, domain=-1.5:0]plot(\x,{sqrt(-\x)});
\draw[samples=150, domain=-1.5:0]plot(\x,-{sqrt(-\x)});
\end{tikzpicture}
\\ \hline
{\rm Type S$'$} & 
$E$ consists of a conic and two lines in a triangle.
&
\begin{tikzpicture}[scale=0.3]
\draw (0,-0.5)--(0,1.5);
\draw (-0.5,0)--(1.5,0);
\draw (-0.3,1.3)--(1.3,-0.3);
\end{tikzpicture}
\\ \hline
{\rm Type T$'$} &  
\begin{minipage}{21em}
$E$ consists of a conic and two lines intersecting in one point.
\end{minipage}
&
\begin{tikzpicture}[scale=0.3]
[baseline=-0.5cm]
\draw (0,-0.5)--(0,1.5);
\draw (-0.5,0)--(1.5,0);
\draw (-1,1)--(1,-1);
\end{tikzpicture}
\\ \hline
{\rm Type FL} &  $E$ is a quadrangle.
&
\begin{tikzpicture}[scale=0.3]
\draw (0,-0.5)--(0,1.5);
\draw (-0.5,0)--(1.5,0);
\draw (1,-0.5)--(1,1.5);
\draw (-0.5,1)--(1.5,1);
\end{tikzpicture}
\\ \hline
{\rm Type WL} &  $E$ is a double conic. 
&
\begin{tikzpicture}[scale=0.3, rotate=135]
\draw (0,-0.5)--(0,1.5);
\draw (0.1,-0.5)--(0.1,1.5);
\end{tikzpicture}
\\ \hline
{\rm Type TWL} & $E$ consists of two double lines.
&
\begin{tikzpicture}[scale=0.2]
\draw (-1,0)--(1,0);
\draw (-1,0.05)--(1,0.05);
\draw (0,-1)--(0,1);
\draw (0.05,-1)--(0.05,1);
\end{tikzpicture}
\\ 
\hline
\end{longtable}
\end{center}
\begin{remark}\label{Rem-MaS}
In \cite{MaS}, for $3$-dimensional cubic AS-regular algebras of Type P, S and T,
the second and third authors gave the complete list of defining relations 
and classified them up to graded algebra isomorphism and graded Morita equivalence. 
\end{remark}
\subsection{Type WL and TWL. }\label{subsec.NR}
When $E$ is of Type WL or TWL, $E$ is not reduced.
\begin{lemma}[{\rm \cite[Lemma 8.19]{ATV2}}]\label{lem.nonreduced}
Let $A=\cA(E,\sigma)$ be a $3$-dimensional cubic AS-regular algebra 
where $E$ is a bidegree $(2,2)$ curve of $\PP^1 \times \PP^1$ 
such that $E$ is not reduced.
 Then $E=2C$, where $C$ is an irreducible curve of bidegree $(1,1)$, or else $C=(\{p\} \times \PP^1) \cup (\PP^1 \times \{p\})$ for some element $p \in \PP^1$.
\end{lemma}

We recall a notion of a twisted algebra $A^{\varphi}$ of a connected graded algebra $A$ by a graded algebra automorphism $\varphi \in \GrAut_k A$, which is formularized by Zhang \cite{Z}.
For $\varphi \in \GrAut_k A$, 
a new graded and associative multiplication $\ast$ on the underlying graded 
vector space $A=\bigoplus_{i \in \NN} A_i$ is defined 
by $a \ast b:=a\varphi^n(b)$ for 
$m,n \in \NN$, 
$a \in A_n, b \in A_m$. 
The graded algebra $(A,\ast)$ is called the {\it twisted algebra} of $A$ by $\varphi$, denoted by $A^{\varphi}$. 

\begin{lemma}[{\rm \cite[Theorems 8.20, 8.29]{ATV2}}]\label{lem.NR}
\begin{enumerate}[{\rm (1)}]
\item Let $A$ be a $3$-dimensional cubic AS-regular algebra of Type WL. 
Then there exists $\varphi \in \GrAut_k A$ such that
$A^{\varphi} \cong B:=k \langle x,y \rangle/(xy^2-2yxy+y^2x, x^2y-2xyx+yx^2)
$. 
\item Let $A$ be a $3$-dimensional cubic AS-regular algebra of Type TWL. 
Then $A \cong k \langle x,y \rangle/(xy^2+y^2x, x^2y+yx^2+y^3)
$. 
\end{enumerate}
\end{lemma}

By Lemma \ref{lem.NR} (2), Type TWL algebra is only one up to graded algebra isomorphism. 

By Lemma \ref{lem.NR} (1) and \cite[Theorem 3.5]{Z}, every Type WL algebra is
graded Morita equivalent to 
$B=k \langle x,y \rangle/(xy^2-2yxy+y^2x, x^2y-2xyx+yx^2)$.
By \cite[Proposition 2.5 (2)]{Z}, $C=A^{\varphi}$ if and only if $A=C^{\varphi^{-1}}$. Thus Lemma \ref{lem.NR} (1) tells us that every Type WL algebra is isomorphic 
to the twisted algebra of $B$ by $\varphi \in \GrAut_k B$.
This means that, to classify Type WL algebras up to graded algebra isomorphism,
it is enough to classify twisted algebras of $B$ by $\varphi \in \GrAut_k B$ up to graded algebra isomorphism.
Note that $B$ is the derivation-quotient algebra $\cD(\om_B)$ where
$$\om_B:=x^2y^2+xy^2x+y^2x^2+yx^2y-2xyxy-2yxyx, 
$$
and $\Aut(\om_B)$ is a subset of $\GrAut_k \cD(\om_B)$.
By direct calculations, $\Aut(\om_B)=\GL_2(k)$.
By \cite[Proposition 5.2 (3)]{MoS1}, 
for any $\varphi \in \Aut(\om_B)$, 
$\cD(\om_B)^{\varphi} \cong \cD(\om_B^{\varphi})$ 
as graded algebras 
where $\om_B^{\varphi}$ is the MS-twist of $\om_B$ by $\varphi$.

\begin{lemma}\label{lem.WL1}
Let $\varphi,\,\psi \in \Aut(\om_B)=\GL_2(k)$. 
Then $\cD(\om_B^{\varphi}) \cong \cD\left(\om_B^{\psi^{-1}\varphi\psi}\right)$
as graded algebras. 
\end{lemma}

\begin{proof}
Let $\varphi, \psi \in \Aut(\om_B)=\GL_2(k)$.
Then there exists $\l \in k\setminus \{0\}$ such that $\psi^{\otimes 4}(\om_B)=\l \om_B$.
The following equation holds: 
	\begin{align*}
	\psi^{\otimes 4}(\om_B^{\psi^{-1}\varphi\psi})
	&=\psi^{\otimes 4}(((\psi^{-1}\varphi\psi)^3\otimes(\psi^{-1}\varphi\psi)^2\otimes(\psi^{-1}\varphi\psi)\otimes{\rm id})(\om_B))\\
	&=\psi^{\otimes 4}(((\psi^{-1}\varphi^3\psi)\otimes(\psi^{-1}\varphi^2\psi)\otimes(\psi^{-1}\varphi\psi)\otimes{\rm id})(\om_B))\\
	&=((\varphi^3\psi)\otimes(\varphi^2\psi)\otimes(\varphi\psi)\otimes \psi))(\om_B) \\
	& =(\varphi^3\otimes\varphi^2\otimes\varphi\otimes{\rm id})(\psi^{\otimes 4}(\om_B))
	=(\varphi^3\otimes\varphi^2\otimes\varphi\otimes{\rm id})(\l\om_B)\\
	&=\l(\varphi^3\otimes\varphi^2\otimes\varphi\otimes{\rm id})(\om_B)
	=\l\om_B^{\varphi}.
	\end{align*}
	By \cite[Lemma 2.10]{MU1}, $\psi$ extends to the isomorphism 
	$\cD\left(\om_B^{\psi^{-1}\varphi\psi}\right) \to \cD(\om_B^{\varphi})$ of graded algebras.
\end{proof}

By \cite[Lemma 4.4]{MaS}, $C$ can be written as 
	$C=C_{\tau}:=\{(p,\tau(p)) \in \PP^1 \times \PP^1 \mid p \in \PP^1\}$ 
for some $\tau \in \Aut_k \PP^1$.

\begin{lemma}\label{lem.WL2}
Let $\varphi, \psi \in \Aut(\om_B)=\GL_2(k)$. 
Then $\cD(\om_B^{\varphi}) \cong \cD\left(\om_B^{\psi}\right)$ as graded algebras
if and only if $\overline{\varphi^{\ast}} \sim \overline{\psi^{\ast}}$ in $\PGL_2(k)$.
\end{lemma}

\begin{proof}
Since $B=\cD(\om_B)=k \langle x,y \rangle/(xy^2-2yxy+y^2x, x^2y-2xyx+yx^2)$,  
it follows from direct calculation that
$\Gamma_{B}=\{(p,p,p) \mid p \in \PP^1\}$. 
This means that $B$ satisfies the condition (G1) in Definition \ref{def.GA}. 
Moreover, $\cP(B)=(C_{\rm id},{\rm id})$.
By \cite[Theorem 3.4 (1)]{MaS}, $\cD(\om_B^{\varphi})$ and $\cD(\om_B^{\psi})$ satisfy the condition (G1) in Definition \ref{def.GA}. 
We set 
$\cP(\cD(\om_B^{\varphi})):=(E_{\varphi},\sigma_{\varphi})$ 
and $\cP(\cD(\om_B^{\psi})):=(E_{\psi},\sigma_{\psi})$.
By \cite[Theorem 3.5 (2)]{MaS}, 
we have $E_{\varphi} \sim C_{\rm id} \sim E_{\psi}$. Moreover, 
$E_{\varphi}=({\rm id} \times \overline{\varphi^{\ast}})(C_{\rm id})=C_{\overline{\varphi^{\ast}}}\, \,\,\text{ and }\,\,\,E_{\psi}=({\rm id} \times \overline{\psi^{\ast}})(C_{\rm id})=C_{\overline{\psi^{\ast}}}$. 
It follows from \cite[Lemma 4.7 (1)]{MaS} that $\cD(\om_B^{\varphi}) \cong \cD\left(\om_B^{\psi}\right)$
if and only if $\overline{\varphi^{\ast}} \sim \overline{\psi^{\ast}}$ in $\PGL_2(k)$.
\end{proof}

\begin{theorem}\label{thm.WL}
Every Type WL algebra is isomorphic 
to one of the following graded algebras{\rm ;}
\begin{align*}
		& \text{\rm (1)}\,B_1:=
		\cD(\om_B^{\varphi_1})
		=k \langle x,y \rangle/(\a^2xy^2+y^2x-2\a yxy, \a^2x^2y+yx^2-2\a xyx)
                  \\
	    &\hspace{15pt} \text{ where } \varphi_1:=
		\begin{pmatrix}
			1 & 0 \\
			0 & \a 
		\end{pmatrix} \quad (\a \in k\setminus \{0\}), \ \text{ or } \\
		&\text{\rm (2)}\,B_2:=\cD(\om_B^{\varphi_2})\\
		&\hspace{20pt} 
		=k \langle x,y \rangle/(xy^2+y^2x-2yxy,\,x^2y+yx^2-2xyx+4xy^2-4yxy+2y^3)
                 \\
		&\hspace{15pt} \text{ where } \varphi_2:=
		\begin{pmatrix}
			1 & 1 \\
			0 & 1
		\end{pmatrix}.
\end{align*}
\end{theorem}
\begin{proof}
Let $\varphi \in \Aut(\om_B)=\GL_2(k)$.
By Lemmas \ref{lem.WL1} and \ref{lem.WL2}, taking the Jordan canonical form of $\varphi$,
it follows that the graded algebra $\cD(\om_B^{\varphi})$ is isomorphic as graded algebras 
to only one of the two graded algebras $B_{1}$ and $B_{2}$ as in the statement.  
	Moreover, 
	by Lemma \ref{lem.WL2}, $B_1=\cD(\om_B^{\varphi_1})$ and $B'_1=\cD(\om_B^{\varphi'_1})$ where $\varphi_1=
	\begin{pmatrix}
		1 & 0 \\
		0 & \a 
	\end{pmatrix}$ and $\varphi'_1=
	\begin{pmatrix}
		1 & 0 \\
		0 & \a'
	\end{pmatrix}$
	\,\,\,$(\a,\,\a' \in k\setminus \{0\})$ are isomorphic as graded algebras if and only if
	$\a'=\a^{\pm 1}$.
\end{proof}
\section{Defining relations of Type S$'$, T$'$ and FL}\label{sec.Class}
For the case that $E$ is reduced, Main theorem in Introduction (Theorems \ref{thm.isom} and \ref{thm.grmod} in Section 4)  are proved by the following six steps:

\begin{description}
	\item[{\bf Step 1}] Classify $E$ up to equivalence and $2$-equivalence.
	\item[{\bf Step 2}] Find all automorphisms $\sigma \in \Aut_k^G E$.
	\item[{\bf Step 3}] Find the defining relations of $\cA(E,\sigma)$ for each $\sigma \in \Aut_k^G E$ by
	using (G2) condition in Definition \ref{def.GA}. 
	\item[{\bf Step 4}] Check AS-regularity of $\cA(E,\sigma)$
	via finding twisted superpotentials. 
	\item[{\bf Step 5}] Classify them up to graded algebra isomorphism in terms of their defining relations
	by using Lemma \ref{thm.GA} (1). 
	\item[{\bf Step 6}] Classify them up to graded Morita equivalence in terms of their defining relations
	by using Lemma \ref{thm.GA} (2). 
\end{description}

In this section, 
we will check {\bf Step 1} to {\bf Step 4} 
as above. 
If a curve $D$ of bidegree $(1,1)$ is reducible, then $D$ is decomposed to two irreducible curves $\{p\} \times \PP^1$ and $\PP^1 \times \{q\}$
for some $p,q \in \PP^1$.
Note that every curve of bidegree $(1,0)$ in $\PP^1 \times \PP^1$ is written as
$\{p\} \times \PP^1$ for some $p \in \PP^1$.
Similarly, every curve of bidegree $(0,1)$ in $\PP^1 \times \PP^1$ is written as
$\PP^1 \times \{q\}$ for some $q \in \PP^1$.
%

\subsection{{\bf Step 1:} Classify $E$ up to equivalence and $2$-equivalence.}
\begin{lemma}\label{lem.step0}
	\begin{enumerate}[{\rm (1)}]
		\item Let $E$ be a union of an irreducible curve $C$ of bidegree $(1,1)$, an irreducible curve $\ell$ of bidegree $(1,0)$ and
		an irreducible curve $\ell'$ of bidegree $(0,1)$
		such that the number of intersections of $E$ is three.
		If $\Aut_k^G E \neq \emptyset$, then
		$$E \sim_2 \PP^1 \times \{(1,0)\} \cup \{(1,0)\} \times \PP^1 \cup C_{\tau}\,\,\,	\text{where } \tau=
		\begin{pmatrix}
			0 & 1 \\
			1 & 0
		\end{pmatrix}
		\vspace{-10pt}.
		$$
		\item Let $E$ be a union of an irreducible curve $C$ of bidegree $(1,1)$, an irreducible curve $\ell$ of bidegree $(1,0)$ and
		an irreducible curve $\ell'$ of bidegree $(0,1)$
		such that the number of intersections of $E$ is only one.
		If $\Aut_k^G E \neq \emptyset$, then $E$ is $2$-equivalent to either
\begin{align*}
& E_1=\PP^1 \times \{(1,0)\} \cup \{(1,0)\} \times \PP^1 \cup C_{\tau_{\a}},\quad 
\ \text{ or } \\
& E_2=\PP^1 \times \{(1,0)\} \cup \{(1,0)\} \times \PP^1 \cup C_{\tau_{1,1}},
\end{align*}
		where $\tau_{\a}=
		\begin{pmatrix}
			1 & 0 \\
			0 & \a
		\end{pmatrix}$ and $\tau_{1,1}=
		\begin{pmatrix}
			1 & 1 \\
			0 & 1
		\end{pmatrix}$. 
		\vspace{-5pt}
\item Let $E$ be a union of two distinct irreducible curves $\ell_1, \ell_2$ of bidegree $(1,0)$ in $\PP^1 \times \PP^1$ and 
		$\ell_3, \ell_4$ of bidegree $(0,1)$ in $\PP^1 \times \PP^1$.
		If $\Aut_k^G E \neq \emptyset$, then
			$$E \sim_2 
			\PP^1 \times \{(1,0)\} \cup \PP^1 \times \{(0,1)\} \cup \{(1,0)\} \times \PP^1 \cup \{(0,1)\} \times \PP^1.$$
	\end{enumerate}
\end{lemma}
%
%

\begin{proof}
(1)\,\,
Let $E$ be a union of an irreducible curve $C_{\tau}$ of bidegree $(1,1)$, an irreducible curve $\ell$ of bidegree $(1,0)$ and
an irreducible curve $\ell'$ of bidegree $(0,1)$
such that the number of intersections of $E$ is three where $\tau \in \Aut_k \PP^1$.
We set $\ell:=\{P_1\} \times \PP^1$ and $\ell':=\PP^1 \times \{P_2\}$
where $P_1, P_2 \in \PP^1$. 
The set of intersections of $E$ is denote by $\{(P_1,P_2)$, $(P_1,\tau(P_1))$, $(\tau^{-1}(P_2), P_2)\}$.  
Since $\tau(P_1) \neq P_2$, there exists $\rho \in \Aut_k \PP^1$ such that
$\rho(\tau(P_1))=(0,1)$ and $\rho(P_2)=(1,0)$.
Since 
$(\rho \times \rho)(E)=(\{\rho(P_1)\} \times \PP^1) \cup (\PP^1 \times \{(1,0)\}) \cup C_{\rho\tau\rho^{-1}}
$, 
the set of intersections of $(\rho \times \rho)(E)$ is denoted by 
$$\{(\rho(P_1),(1,0)), (\rho(P_1),(0,1)), (\rho(\tau^{-1}(P_2)),(1,0))\}.
$$ 
Let $\sigma \in \Aut_k^G ((\rho \times \rho)(E))$ and $(r, (1,0)) \in \PP^1 \times \{(1,0)\}$.
If $\sigma(r,(1,0)) \in \PP^1 \times \{(1,0)\}$, then $r=(1,0)$.
If $\sigma(r,(1,0)) \in C_{\rho\tau\rho^{-1}}$, then $\sigma(r,(1,0))=((1,0),(\rho\tau\rho^{-1})(1,0))$. 
This means that the number of points of $\PP^1 \times \{(1,0)\}$
which satisfies $\sigma(r,(1,0)) \in \PP^1 \times \{(1,0)\}$ or $\sigma(r,(1,0)) \in C_{\rho\tau\rho^{-1}}$ is at most two. 
Therefore, there exists $r \in \PP^1 \setminus \{(1,0)\}$ such that $\sigma(r,(1,0)) \notin C_{\rho\tau\rho^{-1}}$.
Since $\sigma(r,(1,0)) \in \{\rho(P_1)\} \times \PP^1$, we have $\rho(P_1)=(1,0)$.
Since $\sigma$ preserves intersections and $\tau(P_1) \neq P_2$, we have $\rho(\tau^{-1}(P_2))=(0,1)$.
Since 
	$((1,0),(0,1)), ((0,1),(1,0)) \in C_{\rho\tau\rho^{-1}}$,
	$(\rho\tau\rho^{-1})(1,0)=(0,1)$ and $(\rho\tau\rho^{-1})(0,1)=(1,0)$ hold,
so we can write 
$\rho\tau\rho^{-1}=
\begin{pmatrix}
	0 & 1 \\
	\g & 0
\end{pmatrix}$
where $\g \in k \setminus \{0\}$.
Let $\mu=
\begin{pmatrix}
	\g^{\frac{1}{2}} & 0 \\
	0 & 1
\end{pmatrix}\in \Aut_k \PP^1$.
Then $\mu(\rho\tau\rho^{-1})\mu^{-1}=
\begin{pmatrix}
	0 & 1 \\
	1 & 0 
\end{pmatrix}$.
Therefore, $E$ is $2$-equivalent to
$(\{(1,0)\} \times \PP^1) \cup (\PP^1 \times \{(1,0)\}) \cup C_{\tau}$
where $\tau=
\begin{pmatrix}
	0 & 1 \\
	1 & 0
\end{pmatrix} \in \Aut_k \PP^1$.

\smallskip

\noindent
(2)\,\,
Let $E$ be a union of an irreducible curve $C$ of bidegree $(1,1)$, an irreducible curve $\ell$ of bidegree $(1,0)$ and
an irreducible curve $\ell'$ of bidegree $(0,1)$
such that the number of intersections of $E$ is only one.
For $P_1, P_2 \in \PP^1$, we set $\ell:=\{P_1\} \times \PP^1$ and $\ell':=\PP^1 \times \{P_2\}$. 
In this case, the set of
the intersection of $E$ is denoted by
$\{(P_1,P_2)\}$.
Let $\sigma \in \Aut_{k}^G E$. Since $\sigma$ preserves the intersection $(P_1,P_2)$, we have $P_1=P_2$.
Take $\rho \in \Aut_k \PP^1$ with $\rho(P_1)=(1,0)$. In this case, we have  
	$(\rho \times \rho)(E)=(\{(1,0)\} \times \PP^1) \cup (\PP^1 \times \{(1,0)\}) \cup C_{\rho\tau\rho^{-1}}$.
Write $\rho\tau\rho^{-1}=
\begin{pmatrix}
	a & b \\
	c & d 
\end{pmatrix}$. Since $((1,0),(1,0)) \in C_{\rho\tau\rho^{-1}}$, 
$a \neq 0$ and $c=0$ hold, so
$\rho\tau\rho^{-1}=
\begin{pmatrix}
	1 & b \\
	0 & d 
\end{pmatrix}$.
From the above, we may assume that
	$E=(\{(1,0)\} \times \PP^1) \cup (\PP^1 \times \{(1,0)\}) \cup C_{\tau}$
where $\tau=
\begin{pmatrix}
	1 & \b \\
	0 & \a 
\end{pmatrix} \in \Aut_k \PP^1$. 
We will show that $E$ is $2$-equivalent to one of the following; 
	\begin{enumerate}[(i)]
	\item $E_1=\{(1,0)\} \times \PP^1 \cup \PP^1 \times \{(1,0)\} \cup C_{\tau_{\a}},\quad
	\tau_{\a}=
	\begin{pmatrix}
		1 & 0 \\
		0 & \a 
	\end{pmatrix}$, 
	\item $E_2=\{(1,0)\} \times \PP^1 \cup \PP^1 \times \{(1,0)\} \cup C_{\tau_{1,1}},\quad
	\tau_{1,1}=
	\begin{pmatrix}
		1 & 1 \\
		0 & 1
	\end{pmatrix}$.
	\end{enumerate}
When $\b=0$, $E=E_1$, so we assume that $\b \neq 0$.

\noindent
(i)\,\,When $\a \neq 1$, we set 
$\mu:=
\begin{pmatrix}
	1 & \b/(1-\a) \\
	0 & 1
\end{pmatrix} \in \Aut_k \PP^1$. 
In this case, 
	$$\mu\tau\mu^{-1}=
	\begin{pmatrix}
		1 & \b/(1-\a) \\
		0 & 1
	\end{pmatrix}
	\begin{pmatrix}
		1 & \b \\
		0 & \a 
	\end{pmatrix}
	\begin{pmatrix}
		1 & -\b/(1-\a) \\
		0 & 1
	\end{pmatrix}=
	\begin{pmatrix}
		1 & 0 \\
		0 & \a
	\end{pmatrix}=\tau_{\a}.$$ 

\noindent
(ii)\,\,When $\a=1$, we set $\mu:=
\begin{pmatrix}
	1 & 0 \\
	0 & \b 
\end{pmatrix} \in \Aut_k \PP^1$.
In this case,  
	$$\mu\tau\mu^{-1}=
	\begin{pmatrix}
		1 & 0 \\
		0 & \b 
	\end{pmatrix}
	\begin{pmatrix}
		1 & \b \\
		0 & 1 
	\end{pmatrix}
	\begin{pmatrix}
		1 & 0 \\
		0 & \b^{-1}
	\end{pmatrix}=
	\begin{pmatrix}
		1 & 1 \\
		0 & 1 
	\end{pmatrix}=\tau_{1,1}.$$
Therefore, $E$ is $2$-equivalent to either $E_1$ or $E_2$.

\smallskip

\noindent
(3)\,\,Let $E$ be a union of two distinct irreducible curves $\ell_1, \ell_2$ of bidegree $(1,0)$ in $\PP^1 \times \PP^1$ and 
$\ell_3, \ell_4$ of bidegree $(0,1)$ in $\PP^1 \times \PP^1$.
We set $\ell_1:=\{P_1\} \times \PP^1$, $\ell_2:=\{P_2\} \times \PP^1$, $\ell_3:=\PP^1 \times \{P_3\}$ and $\ell_4:=\PP^1 \times \{P_4\}$
where $P_1, P_2, P_3, P_4 \in \PP^1$, $P_1 \neq P_2$ and $P_3 \neq P_4$.
Let $\sigma \in \Aut_k^G E$. 
\begin{itemize}
\item If $\sigma(p,P_3) \in \ell_3$, then $\sigma(p,P_3)=(P_3,P_3)$. 
\item If $\sigma(p,P_3) \in \ell_4$, then $\sigma(p,P_3)=(P_3,P_4)$.
This means that there exists $(p,P_3) \in \ell_3$ such that $\sigma(p,P_3) \in \ell_1$ or $\sigma(p,P_3) \in \ell_2$. 
\item If $\sigma(p,P_3) \in \ell_1$ (resp. $\sigma(p,P_3) \in \ell_2$), then $P_3=P_1$ (resp. $P_3=P_2$).
Also, there exists $(p,P_4) \in \ell_4$ such that $\sigma(p,P_4) \in \ell_1$ or $\sigma(p,P_4) \in \ell_2$.
\item If $\sigma(p,P_4) \in \ell_1$ (resp. $\sigma(p,P_4) \in \ell_2$), then $P_4=P_1$ (resp. $P_4=P_2$).
Since $P_3 \neq P_4$, we have $(P_3,P_4)=(P_1,P_2)$ or $(P_3,P_4)=(P_2,P_1)$. 
\end{itemize}

\noindent
From the above, we may assume that
	$E=(\{P_1\} \times \PP^1) \cup (\{P_2\} \times \PP^1) \cup (\PP^1 \times \{P_1\}) \cup (\PP^1 \times \{P_2\})$.
Since there exists $\tau \in \Aut_k \PP^1$ such that $\tau(P_1)=(1,0)$ and $\tau(P_2)=(0,1)$,
$E$ is $2$-equivalent to
	$(\{(1,0)\} \times \PP^1) \cup (\{(0,1)\} \times \PP^1) \cup (\PP^1 \times \{(1,0)\}) \cup (\PP^1 \times \{(0,1)\})$.
\end{proof}



\subsection{{\bf Step 2:} Find all automorphisms $\sigma \in \Aut_k^G E$.}
\begin{lemma}\label{lem.step1}
	\begin{enumerate}[{\rm (1)}]
		\item Let $E=\{P\} \times \PP^1 \cup \PP^1 \times \{P\} \cup C_{\tau}$
		where $P=(1,0)$ and $\tau=
		\begin{pmatrix}
			0 & 1 \\
			1 & 0
		\end{pmatrix}$.
		Then
		every automorphism $\sigma \in \Aut_k^G E$ is written as one of the following{\rm :}
			$${\rm (i)\,}\begin{cases}
				\sigma(p,P)=(P,\tau_{\a}(p)), \\
				\sigma(P,p)=(p,P), \\
				\sigma(p,\tau(p))=(\tau(p),p),
			\end{cases}
			\quad
			{\rm \quad (ii)\,}\begin{cases}
				\sigma(p,P)=(P,\mu_{\a}(p)), \\
				\sigma(P,p)=(p,\tau(p)), \\
				\sigma(p,\tau(p))=(\tau(p),P), 
			\end{cases}$$
		where $\tau_{\a}=
		\begin{pmatrix}
			1 & 0 \\
			0 & \a
		\end{pmatrix}$, $\mu_{\a}=
		\begin{pmatrix}
			0 & 1 \\
			\a & 0 
		\end{pmatrix}$ and $\a \in k \setminus \{0\}$. 
		
		\item 
		\begin{itemize}
		\item Let $E=\{P\} \times \PP^1 \cup \PP^1 \times \{P\} \cup C_{\tau_{\a}}$
		where $P=(1,0)$ and $\tau_{\a}=
		\begin{pmatrix}
			1 & 0 \\
			0 & \a
		\end{pmatrix}$.
		Then every automorphism $\sigma \in \Aut_k^G E$ is written as one of the following{\rm :}
			$${\rm (i)\,}\begin{cases}
				\sigma(p,P)=(P,\tau_{\b,\g}(p)), \\
				\sigma(P,p)=(p,P), \\
				\sigma(p,\tau_{\a}(p))=(\tau_{\a}(p),\tau_{\a}^2(p)),
			\end{cases}\quad
			{\rm (ii)\,}\begin{cases}
				\sigma(p,P)=(P,\tau_{\b,\g}(p)), \\
				\sigma(P,p)=(p,\tau_{\a}(p)), \\
				\sigma(p,\tau_{\a}(p))=(\tau_{\a}(p),P), 
			\end{cases}$$
		where $\tau_{\b,\g}=
		\begin{pmatrix}
			1 & \b \\
			0 & \g
		\end{pmatrix}$ and $\b\in k$, $\g \in k \setminus \{0\}$.
		
		\item Let $E=\{P\} \times \PP^1 \cup \PP^1 \times \{P\} \cup C_{\tau_{1,1}}$
		where $P=(1,0)$ and $\tau_{1,1}=
		\begin{pmatrix}
			1 & 1 \\
			0 & 1
		\end{pmatrix}$.
		Then every automorphism $\sigma \in \Aut_k^G E$ is written as one of the following{\rm :}
			$${\rm (i)\,}\begin{cases}
				\sigma(p,P)=(P,\tau_{\b,\g}(p)), \\
				\sigma(P,p)=(p,P), \\
				\sigma(p,\tau_{1,1}(p))=(\tau_{1,1}(p),\tau_{1,1}^2(p)),
			\end{cases}
			{\rm (ii)\,}\begin{cases}
				\sigma(p,P)=(P,\tau_{\b,\g}(p)), \\
				\sigma(P,p)=(p,\tau_{1,1}(p)), \\
				\sigma(p,\tau_{1,1}(p))=(\tau_{1,1}(p),P), 
			\end{cases}$$
		where $\tau_{\b,\g}=
		\begin{pmatrix}
			1 & \b \\
			0 & \g
		\end{pmatrix}$ and $\b\in k$, $\g \in k \setminus \{0\}$.
		\end{itemize}
		
		\item Let $E=\PP^1 \times \{P\} \cup \PP^1 \times \{Q\} \cup \{P\} \times \PP^1 \cup \{Q\} \times \PP^1$ where $P=(1,0)$ and $Q=(0,1)$.
		Then every automorphism $\sigma \in \Aut_k^G E$ is written as one of the following{\rm :}
			$${\rm (i)\,}\begin{cases}
				\sigma(p,P)=(P,\tau_{\a}(p)), \\
				\sigma(p,Q)=(Q,\tau_{\b}(p)), \\
				\sigma(P,p)=(p,P), \\
				\sigma(Q,p)=(p,Q),
			\end{cases}\quad
			{\rm (ii)\,}\begin{cases}
				\sigma(p,P)=(P,\mu_{\a}(p)), \\
				\sigma(p,Q)=(Q,\mu_{\b}(p)), \\
				\sigma(P,p)=(p,Q), \\
				\sigma(Q,p)=(p,P), 
			\end{cases}$$
		where $\tau_{\a}=
		\begin{pmatrix}
			1 & 0 \\
			0 & \a
		\end{pmatrix}$, $\tau_{\b}=
		\begin{pmatrix}
		1 & 0 \\
		0 & \b
		\end{pmatrix}$, $\mu_{\a}=
		\begin{pmatrix}
			0 & 1 \\
			\a & 0 
		\end{pmatrix}$, $\mu_{\b}=
		\begin{pmatrix}
		0 & 1 \\
		\b & 0 
		\end{pmatrix}$ and $\a, \b \in k$, $\a\b \neq 0$.
	\end{enumerate}
\end{lemma}
%
\vspace{-15pt}

\begin{proof}
(1)\,\, 
Let $E=\{P\} \times \PP^1 \cup \PP^1 \times \{P\} \cup C_{\tau}$
where $P=(1,0)$, $\tau=
\begin{pmatrix}
	0 & 1 \\
	1 & 0
\end{pmatrix}$.

\noindent
(i) Assume that $\sigma(\PP^1 \times \{P\})=\{P\} \times \PP^1$, $\sigma(\{P\} \times \PP^1)= \PP^1 \times \{P\}$ and $\sigma(C_{\tau})=C_{\tau}$.
In this case, $\sigma$ is written as 
$\begin{cases}
	\sigma(p,P)=(P,\rho(p)), \\
	\sigma(P,p)=(p,P), \\
	\sigma(p,\tau(p))=(\tau(p),\tau^2(p)).
\end{cases}$
Since $\sigma(P,P)=(P,P)$ and $\sigma(Q,P)=(P,Q)$, we have
$\rho(P)=P$, $\rho(Q)=\tau(Q)$. So, we can write $\rho=
\begin{pmatrix}
	1 & 0 \\
	0 & \a
\end{pmatrix}$\,\,
($\a \in k \setminus \{0\}$).

\noindent
(ii) Assume that $\sigma(\{P\} \times \PP^1)=C_{\tau_0}$, $\sigma(\PP^1 \times \{P\})=\{P\} \times \PP^1$ and $\sigma(C_{\tau_0})=\PP^1 \times \{P\}$.
In this case, $\sigma$ is written as 
$\begin{cases}
	\sigma(p,P)=(P,\rho(p)), \\
	\sigma(P,p)=(p,\tau(p)), \\
	\sigma(p,\tau(p))=(\tau(p),P).
\end{cases}$
Since $\sigma(P,P)=(P,Q)$ and $\sigma(Q,P)=(P,P)$, we have 
$\rho(P)=Q$, $\rho(Q)=P$. So, we can write
$\rho=
\begin{pmatrix}
	0 & 1 \\
	\a & 0
\end{pmatrix}$ \,\,($\a \in k \setminus \{0\}$).
\noindent
(2)\,\,
Let $E=\{P\} \times \PP^1 \cup \PP^1 \times \{P\} \cup C_{\tau_{\a}}$
where $\tau_{\a}=
\begin{pmatrix}
	1 & 0 \\
	0 & \a
\end{pmatrix}$.

\noindent
(i)\,\,Assume that $\sigma(\{P\} \times \PP^1)=\PP^1 \times \{P\}$, $\sigma(\PP^1 \times \{P\})=(\{P\} \times \PP^1$ and $\sigma(C_{\tau_{\a}})=C_{\tau_{\a}}$.
In this case, $\sigma$ is written as 
$\begin{cases}
	\sigma(p,P)=(P,\rho(p)), \\
	\sigma(P,p)=(p,P), \\
	\sigma(p,\tau_{\a}(p))=(\tau_{\a}(p),\tau_{\a}^2(p)).
\end{cases}$
Since $\sigma(P,P)=(P,P)$, we have $\rho(P)=P$. 
So we can write
$\rho=
\begin{pmatrix}
	1 & \b \\
	0 & \g 
\end{pmatrix}$ \,\,($\b\in k$, $\g \in k \setminus \{0\}$).

\smallskip

\noindent
(ii)\,\,Assume that $\sigma(\{P\} \times \PP^1)=C_{\tau_{\a}}$, $\sigma(\PP^1 \times \{P\})=(\{P\} \times \PP^1$ and $\sigma(C_{\tau_{\a}})=\PP^1 \times \{P\}$.
In this case, $\sigma$ is written as 
$\begin{cases}
	\sigma(p,P)=(P,\rho(p)), \\
	\sigma(P,p)=(p,\tau_{\a}(p)), \\
	\sigma(p,\tau_{\a}(p))=(\tau_{\a}(p),P).
\end{cases}$
Since $\sigma(P,P)=(P,P)$, we have $\rho(P)=P$. So we can write
$\rho=
\begin{pmatrix}
	1 & \b \\
	0 & \g 
\end{pmatrix}$ 
\,\,($\b\in k$, $\g \in k \setminus \{0\}$).

Let $E=\{P\} \times \PP^1 \cup \PP^1 \times \{P\} \cup C_{\tau_{1,1}}$
where $\tau_{1,1}=
\begin{pmatrix}
	1 & 1 \\
	0 & 1
\end{pmatrix}$.

\noindent
(i)\,\,Assume that $\sigma((\{P\} \times \PP^1)=\PP^1 \times \{P\}$, $\sigma(\PP^1 \times \{P\})=\{P\} \times \PP^1$ and $\sigma(C_{\tau_{1,1}})=C_{\tau_{1,1}}$.
In this case, $\sigma$ is written as 
$\begin{cases}
	\sigma(p,P)=(P,\rho(p)), \\
	\sigma(P,p)=(p,P), \\
	\sigma(p,\tau_{1,1}(p))=(\tau_{1,1}(p),\tau_{1,1}^2(p)).
\end{cases}$
Since $\sigma(P,P)=(P,P)$, we have $\rho(P)=P$. 
So we can write
$\rho=
\begin{pmatrix}
	1 & \b \\
	0 & \g 
\end{pmatrix}$
\,\,($\b\in k$, $\g \in k \setminus \{0\}$).

\noindent
(ii)\,\,Assume that $\sigma(\{P\} \times \PP^1)=C_{\tau_{1,1}}$, $\sigma(\PP^1 \times \{P\})=(\{P\} \times \PP^1$ and $\sigma(C_{\tau_{1,1}})=\PP^1 \times \{P\}$.
In this case, $\sigma$ is written as 
$\begin{cases}
	\sigma(p,P)=(P,\rho(p)), \\
	\sigma(P,p)=(p,\tau_{1,1}(p)), \\
	\sigma(p,\tau_{1,1}(p))=(\tau_{1,1}(p),P).
\end{cases}$
Since $\sigma(P,P)=(P,P)$, we have $\rho(P)=P$. 
So we can write
$\rho=
\begin{pmatrix}
	1 & \b \\
	0 & \g 
\end{pmatrix}$ 
\,($\b\in k$, $\g \in k \setminus \{0\}$).

\noindent 
(3)\,\,Let $E=(\{P\} \times \PP^1) \cup (\{Q\} \times \PP^1) \cup (\PP^1 \times \{P\}) \cup (\PP^1 \times \{Q\})$
We also use the following notations:
\begin{align*}
	\ell_1:=\{P\} \times \PP^1, \ell_2:=\{Q\} \times \PP^1, 
	\ell_3:=\PP^1 \times \{P\}, \ell_4:=\PP^1 \times \{Q\}.
\end{align*}
Let $\sigma \in \Aut_k^G E$. Then
$\sigma(\ell_3)=\ell_1$ and $\sigma(\ell_4)=\ell_2$.
Moreover, we can write
$\begin{cases}
	\sigma(p,P)=(P,\rho(p)), \\
	\sigma(p,Q)=(Q,\rho'(p)),
\end{cases}$
for $\rho, \rho' \in \Aut_k \PP^1$.

\smallskip

\noindent
(i)\,\,Assume that $\sigma(\ell_1)=\ell_3, \sigma(\ell_2)=\ell_4, \sigma(\ell_3)=\ell_1, \sigma(\ell_4)=\ell_2$.
In this case, $\sigma$ is written as 
$\begin{cases}
	\sigma(P,p)=(p,P), \\
	\sigma(Q,p)=(p,Q), \\
	\sigma(p,P)=(P,\rho(p)), \\
	\sigma(p,Q)=(Q,\rho'(p)).
\end{cases}$
Since
$\sigma(P,P)=(P,P)$, $\sigma(P,Q)=(Q,P)$, 
$\sigma(Q,P)=(P,Q)$, $\sigma(Q,Q)=(Q,Q)$,
we have 
$
	\rho(P)=P,\quad \rho(Q)=Q, \rho'(P)=P,\quad \rho'(Q)=Q
$. 
So we can write $\rho=
\begin{pmatrix}
	1 & 0 \\
	0 & \a
\end{pmatrix}$ and $\rho'=
\begin{pmatrix}
	1 & 0 \\
	0 & \b 
\end{pmatrix}$ 
\,\,($\a, \b \in k$, $\a\b \neq 0$).

\smallskip

\noindent
(ii)\,\,Assume that $\sigma(\ell_1)=\ell_4, \sigma(\ell_2)=\ell_3, \sigma(\ell_3)=\ell_1, \sigma(\ell_4)=\ell_2$.
In this case, $\sigma$ is written as follows:
$\begin{cases}
	\sigma(P,p)=(p,Q), \\
	\sigma(Q,p)=(p,P), \\
	\sigma(p,P)=(P,\rho(p)), \\
	\sigma(p,Q)=(Q,\rho'(p)).
\end{cases}$
Since
	$\sigma(P,P)=(P,Q)$, $\sigma(P,Q)=(Q,Q)$, 	
	$\sigma(Q,P)=(P,P)$, $\sigma(Q,Q)=(Q,P)$,
we have 
	$\rho(P)=Q$, $\rho(Q)=P$, 		
	$\rho'(P)=Q$, $\rho'(Q)=P$,
so we can write $\rho=
\begin{pmatrix}
	0 & 1 \\
	\a & 0
\end{pmatrix}$ and $\rho'=
\begin{pmatrix}
	0 & 1 \\
	\b & 0 
\end{pmatrix}$
\,\,($\a, \b \in k$, $\a\b \neq 0$).
\end{proof}

\subsection{{\bf Step 3:} Find the defining relations of $\cA(E,\sigma)$ for each $\sigma \in \Aut_k^G E$. }



\begin{theorem}\label{thm.DR}
	Let $A=\kxy/(g_1,g_2)=\cA(E,\sigma)$ be a $3$-dimensional cubic AS-regular algebra
	of Type S\,$'$, T\,$'$ or FL. 
	Then {\rm Table 1} gives the list of defining relations $g_1, g_2$ and conditions.
	Moreover, Type T\,$'$ is further divided into Type T\,$'_{1}$ and Type T\,$'_{2}$
	in terms of the form of $E$,
	and Type FL is further divided into Type FL$_{1}$ and Type FL$_{2}$
	in terms of the form of $\sigma$.

	\begin{center}
		{\renewcommand\arraystretch{1.1}
			{\small
				\begin{longtable}{|p{1.0cm}|p{8.5cm}|p{2.8cm}|}
					\multicolumn{3}{c}{{\rm Table 1: List of defining relations $g_1,\,g_2$, and conditions}}
					\\ \hline
					{\rm Type} & {\rm Defining relations $g_1$ and $g_2$} & {\rm Conditions} \\ \hline\hline
					{\rm S$'$} &
					$\begin{cases}
						g_1=x^2y-\a yx^2+(\a-1)y^3, \\
						g_2=xy^2-y^2x
					\end{cases}$ &
					$\a \in k\setminus \{0\}$
					\\ \hline
					{\rm T$'_{1}$} &
					$\begin{cases}
						g_1=x^2y-\d^2 yx^2+\a yxy-\a\d y^2x, \\
						g_2=xy^2-\d^2y^2x
					\end{cases}$ &
					$\a \in k$, $\d \in k\setminus \{0\}$
					\\ \hline
					{\rm T$'_{2}$} &
					$\begin{cases}
						g_1=x^2y-yx^2+\a yxy+(2-\a)y^2x+(\a-2)y^3, \\
						g_2=xy^2-y^2x+2y^3
					\end{cases}$ &
					$\a \in k$
					\\ \hline
					{\rm FL$_{1}$} &
					$\begin{cases}
						g_1=x^2y-\a yx^2, \\
						g_2=xy^2-\b y^2x
					\end{cases}$ &
					$\a, \b \in k$, $\a\b \neq 0$
					\\ \hline
					{\rm FL$_{2}$} &
					$\begin{cases}
						g_1=yxy-\a x^3, \\
						g_2=\b xyx-y^3
					\end{cases}$ &
					$\a,\b \in k$, $\a\b \neq 0$
					\\ \hline
				\end{longtable}
		}}
	\end{center}
\end{theorem}
\begin{proof}
	Let $g=a_1x^3+a_2x^2y+a_3xyx+a_4yx^2+a_5xy^2+a_6yxy+a_7y^2x+a_8y^3$
	be a homogeneous polynomial of $\kxy$ of degree $3$, and $P=(1,0), Q=(0,1) \in \PP^1$.
	For any $(p,q) \in E$, assume that $g(p,\sigma(p,q))=0$. 
	
	
	\smallskip
	
	\noindent
	(1) (Type S$'$)\,\,
	We prove by the similar way of Type T$_{1}'$ as in (2-1). 

	\smallskip
	
	\noindent
	(2-1) (Type T$_{1}'$)\,\,
	Let $E=\{P\} \times \PP^1 \cup \PP^1 \times \{P\} \cup C_{\tau_{\a}}$ and $\tau_{\a}=
	\begin{pmatrix}
		1 & 0 \\
		0 & \a
	\end{pmatrix}$. Assume that $\sigma$ is given by 
		$$\begin{cases}
			\sigma(p,P)=(P,\tau_{\b,\g}(p)), \\
			\sigma(P,p)=(p,P), \\
			\sigma(p,\tau_{\a}(p))=(\tau_{\a}(p),\tau_{\a}^2(p)),
		\end{cases}\,
		\tau_{\b,\g}=
		\begin{pmatrix}
			1 & \b \\
			0 & \g 
		\end{pmatrix}\quad(\b\in k,\,\g \in k\setminus \{0\}).$$ 
	In this case, we have 
	$$
	\begin{cases}
		0=g(P,\sigma(P,P))=g(P,P,\tau_{\b,\g}(P))=g(P,P,P)=a_1, \\
		0=g(Q,\sigma(Q,P))=g(Q,P,\tau_{\b,\g}(Q))=a_4\b+a_6\g, \\
		0=g((1,1),\sigma((1,1),P))=g((1,1),P,(1+\b,\g))=a_2\g+a_4, \\
		0=g(P,\sigma(P,Q))=g(P,Q,P)=a_3, \\
		0=g(Q,\sigma(Q,\tau_{\a}(Q))=g(Q,Q,Q)=a_8.
	\end{cases}
	$$
	For $p=(1,\l)$ with $\l \neq 0$, $\tau_{\a}(p)=(1,\l\a)$, $\tau_{\a}^2(p)=(1,\l\a^2)$ hold.
	We have 
$$
		0=g(p,\sigma(p,\tau_{\a}(p)))=g(p,\tau_{\a}(p),\tau_{\a}^2(p))
		=(a_5\a^3+a_6\a^2+a_7\a)\l^2+(a_2\a^2+a_4)\l,
$$
	so $a_5\a^2+a_6\a+a_7=0$ and $a_2\a^2+a_4=0$. If $\g-\a^2 \neq 0$, then $a_2=0$.
Hence, $g=a_5(xy^2-\a^2y^2x)$, so this contradicts.
	When $\g-\a^2=0$, we have 
		$g=a_5(xy^2-\a^2y^2x)+a_6(x^2y-\a^2 yx^2+\b yxy-\a\b y^2x)$. 
	
	Next, assume that $\sigma$ is given by 
		$$\begin{cases}
			\sigma(p,P)=(P,\tau_{\b,\g}(p)), \\
			\sigma(P,p)=(p,\tau_{\a}(p)), \\
			\sigma(p,\tau_{\a}(p))=(\tau_{\a}(p),P), 
		\end{cases}\,
		\tau_{\b,\g}=
		\begin{pmatrix}
			1 & \b \\
			0 & \g 
		\end{pmatrix}\quad(\b\in k, \g \in k\setminus \{0\}).$$ 
	In this case, we have 
	$$
	\begin{cases}
		0=g(P,\sigma(P,P))=g(P,P,\tau_{\b,\g}(P))=g(P,P,P)=a_1, \\
		0=g(Q,\sigma(Q,P))=g(Q,P,\tau_{\b,\g}(Q))=a_4\b+a_6\g, \\
		0=g((1,1),\sigma((1,1),P))=g((1,1),P,(1+\b,\g))=a_2\g+a_4, \\
		0=g(P,\sigma(P,Q))=g(P,Q,(1,1))=a_5, \\
		0=g(P,\sigma(P,(1,1)))=g(P,(1,1),(2,1))=a_2\a+a_3, \\
		0=g(Q,\sigma(Q,\tau_{\a}(Q)))=g(Q,\tau_{\a}(Q),P)=a_7, \\
		0=g((1,1),\sigma((1,1),\tau_{\a}(1,1)))=g((1,1),\tau_{\a}(1,1),P)=a_3\a+a_4.
	\end{cases}
	$$
	If $\g+\a^2 \neq 0$, then $a_2=0$. 
Hence, we have $g=a_8y^3$, so this contradicts.
	When $\g+\a^2=0$, we have 
		$g=a_2(x^2y+\a^2yx^2-\a xyx+\b yxy)+a_8y^3$. 
	Since $\cA(E,\sigma)$ is not a domain,
	$\cA(E,\sigma)$ does not become a $3$-dimensional cubic AS-regular algebra.
	
	\medskip
	
	\noindent
	(2-2) (Type T$_{2}'$)\,\,
	Let $E=\{P\} \times \PP^1 \cup \PP^1 \times \{P\} \cup C_{\tau_{1,1}}$ and $\tau_{1,1}=
	\begin{pmatrix}
		1 & 1 \\
		0 & 1
	\end{pmatrix}$. Assume that $\sigma$ is given by 
		$$\begin{cases}
			\sigma(p,P)=(P,\tau_{\b,\g}(p)), \\
			\sigma(P,p)=(p,P), \\
			\sigma(p,\tau_{1,1}(p))=(\tau_{1,1}(p),\tau_{1,1}^2(p)),
		\end{cases}\,
		\tau_{\b,\g}=
		\begin{pmatrix}
			1 & \b \\
			0 & \g 
		\end{pmatrix}\quad(\b\in k, \g \in k\setminus \{0\}).$$ 
	In this case, we have 
	$$
	\begin{cases}
		0=g(P,\sigma(P,P))=g(P,P,\tau(P))=g(P,P,P)=a_1, \\
		0=g(Q,\sigma(Q,P))=g(Q,P,(\b,\g))=a_4\b+a_6\g, \\
		0=g((1,1),\sigma((1,1),P))=g((1,1),P,(1+\b,\g))=a_2\g+a_4, \\
		0=g(P,\sigma(P,Q))=g(P,Q,P)=a_3.
	\end{cases}
	$$
	For $p=(1,\l)$ with $\l \neq 0$, $\tau_{1,1}(p)=(1+\l,\l)$ and $\tau_{1,1}^2(p)=(1+2\l,\l)$ hold.
	We have 
	\vspace{-5pt}
	\begin{align*}
		0&=g(p,\sigma(p,\tau_{1,1}(p)))=g(p,\tau_{1,1}(p),\tau_{1,1}^2(p)) \\
		&=(a_4(2-\b\g^{-1})+2a_7+a_8)\l^2+(a_4(-\b\g^{-1}-\g^{-1}+3)+a_5+a_7)\l 
		+a_4(-\g^{-1}+1),
	\end{align*}
	so $a_4(2-\b\g^{-1})+2a_7+a_8=0$, $a_4(-\b\g^{-1}-\g^{-1}+3)+a_5+a_7=0$ and $a_4(-\g^{-1}+1)=0$.
	If $-\g^{-1}+1 \neq 0$, then $a_4=0$. In this case, we have $g=a_5(xy^2-y^2x+2y^3)$, so
	this contradicts. When $-\g^{-1}+1=0$, that is, $\g=1$,
		$a_4(2-\b)+a_5+a_7=0$, 
		$a_4(2-\b)+2a_7+a_8=0$.
	Therefore, we have 
		$$g=a_4(-x^2y+yx^2-\b yxy+(\b-2)y^2x+(2-\b)y^3)+a_5(xy^2-y^2x+2y^3).$$
	
	Next, assume that $\sigma$ is given by  
		$$\begin{cases}
			\sigma(p,P)=(P,\tau_{\b,\g}(p)), \\
			\sigma(P,p)=(p,\tau_{1,1}(p)), \\
			\sigma(p,\tau_{1,1}(p))=(\tau_{1,1}(p),P),
		\end{cases}\,
		\tau_{\b,\g}=
		\begin{pmatrix}
			1 & \b \\
			0 & \g 
		\end{pmatrix}
		\quad
		(\b\in k, \g \in k \setminus \{0\}).
		$$ 
	In this case, we have 
	$$
	\begin{cases}
		0=g(P,\sigma(P,P))=g(P,P,\tau_{\b,\g}(P))=g(P,P,P)=a_1, \\
		0=g(Q,\sigma(Q,P))=g(Q,P,(\b,\g))=a_4\b+a_6\g, \\
		0=g((1,1),\sigma((1,1),P))=g((1,1),P,(1+\b,\g))=a_2\g+a_4, \\
		0=g(P,\sigma(P,Q))=g(P,Q,(1,1))=a_3+a_5, \\
		0=g(P,\sigma(P,(1,1)))=g(P,(1,1),(2,1))=a_3-a_4\g^{-1}, \\
		0=g(Q,\sigma(Q,\tau_{1,1}(Q)))=g(Q,\tau_{1,1}(Q),P)=a_4+a_7, \\
		0=g((1,1),\sigma((1,1),\tau_{1,1}(1,1)))=g((1,1),\tau_{1,1}(1,1),P)=a_3+a_4.
	\end{cases}
	$$
	If $\g^{-1}+1\neq 0$, then $a_4=0$. Hence, we have $g=a_8y^3$, so this contradicts.
	When $\g=-1$, we have 
		$g=a_4(x^2y+yx^2+\b yxy-xyx+xy^2-y^2x)+a_8y^3$.
	Since $\cA(E,\sigma)$ is not a domain, it does not become AS-regular.
    
    \smallskip
    
    \noindent
    (3-1) (Type FL$_{1}$)\,\,
    Let $E=\{P\} \times \PP^1 \cup \{Q\} \times \PP^1 \cup \PP^1 \times \{P\} \cup \PP^1 \times \{Q\}$. Assume that $\sigma$ is given by 
    	$$\begin{cases}
    		\sigma(P,p)=(p,P), \\
    		\sigma(Q,p)=(p,Q), \\
    		\sigma(p,P)=(P,\tau_{\a}(p)), \\
    		\sigma(p,Q)=(Q,\tau_{\b}(p)),
    	\end{cases}\,
    	\tau_{\a}=
    	\begin{pmatrix}
    		1 & 0 \\
    		0 & \a
    	\end{pmatrix},\,\tau_{\b}=
    	\begin{pmatrix}
    		1 & 0 \\
    		0 & \b 
    	\end{pmatrix}.$$
    In this case, we have 
    $
    \begin{cases}
    	0=g(P_1,\sigma(P,P))=g(P,P,P)=a_1, \\
    	0=g(P_1,\sigma(P,Q))=g(P,Q,P)=a_3, \\
    	0=g(Q,\sigma(Q,P))=g(Q,P,Q)=a_6, \\
    	0=g(Q,\sigma(Q,Q))=g(Q,Q,Q)=a_8.
    \end{cases}
    $
    
    \noindent
    For $p=(1,1) \in \PP^1$, we have 
    $
    \begin{cases}
    	0=g(p,\sigma(p,P))=g(p,P,\tau_{\a}(p))=a_2\a+a_4, \\
    	0=g(p,\sigma(p,Q))=g(p,Q,\tau_{\b}(p))=a_5\b+a_7.
    \end{cases}
    $
    
    \noindent
    Therefore, we have
    	$g=a_2(x^2y-\a yx^2)+a_5(xy^2-\b y^2x)$. 
    
    \smallskip
    
    \noindent
    (3-2) (Type FL$_{2}$)\quad
    Assume that $\sigma$ is given by 
    	$$\begin{cases}
    		\sigma(P,p)=(p,Q), \\
    		\sigma(Q,p)=(p,P), \\
    		\sigma(p,P)=(P,\mu_{\a}(p)), \\
    		\sigma(p,Q)=(Q,\mu_{\b}(p)),
    	\end{cases}\,
    	\mu_{\a}=
    	\begin{pmatrix}
    		0 & 1 \\
    		\a & 0
    	\end{pmatrix},\,
	    \mu_{\b}=
    	\begin{pmatrix}
    		0 & 1 \\
    		\b & 0
    	\end{pmatrix}.
	$$
    In this case, we have 
    $
    \begin{cases}
    	0=g(P,\sigma(P,P))=g(P,P,Q)=a_2, \\
    	0=g(P,\sigma(P,Q))=g(P,Q,Q)=a_5, \\
    	0=g(Q,\sigma(Q,P))=g(Q,P,P)=a_4, \\
    	0=g(Q,\sigma(Q,Q))=g(Q,Q,P)=a_7.
    \end{cases}
    $ 
    
    \noindent
    For $p=(1,1) \in \PP^1$, we have 
    $
    \begin{cases}
    	0=g(p,\sigma(p,P))=g(p,P,\mu_{\a}(p))=a_1+a_6\a, \\
    	0=g(p,\sigma(p,Q))=g(p,Q,\mu_{\b}(p))=a_3+a_8\b.
    \end{cases}
    $
    
    \noindent
    Therefore, 
    	$g=a_6(yxy-\a x^3)+a_8(-\b xyx+y^3)$.
\end{proof}

\subsection{{\bf Step 4:} Check AS-regularity of $\cA(E,\sigma)$. }



\begin{proposition}\label{prop.P}
	Let $X \in \{\text{S}\,', \text{T}\,'_1, \text{T}\,'_2, \text{FL}_1,\text{FL}_2\}$. Then 
		every Type X algebra is isomorphic to $\cD(\om)$ where a potential $\om$ is in {\rm Table 2}.
		Also, every potential $\om$ listed in Table $2$ is a twisted superpotential, 
and its derivation-quotient algebra $\mathcal{D}(\omega)$ is AS-regular. 
	
	\begin{center}
		{\renewcommand\arraystretch{1.1}
			{\small
				\begin{longtable}{|p{1.0cm}|p{7.0cm}|p{3.0cm}|}
					\multicolumn{3}{c}
					{{\rm Table 2: List of potentials $\om$ and conditions}}
					\\ 
					\hline
					{\rm Type} & {\rm Potentials $\om$} & {\rm Conditions} \\ \hline\hline
					{\rm S$'$} &
					$x^2y^2+yx^2y-xy^2x+y^2x^2-2y^4$ &
					\hfill \textnormal{---------------------} \hfill \rule{0pt}{10pt}
					\\ \hline
					{\rm T$'_1$} & 
					$x^2y^2-yx^2y-xy^2x+y^2x^2-\a y^2xy+\a  yxy^2$ &
					$\a \neq 0$
					\\ \hline
					{\rm T$'_2$} & 
					$x^2y^2-yx^2y-xy^2x+y^2x^2+2xy^3+\a yxy^2-\a y^2xy-2y^3x+(\a+2)y^4$ &
					$\a \neq 2$
					\\ \hline
					{\rm FL$_1$} &
					$x^2y^2-\a yx^2y+\a xy^2x+\a^2 y^2x^2$ &
					$\a \neq 0$
					\\ \hline
					{\rm FL$_2$} & 
					$-\a\b x^4+\b xyxy+\b yxyx-y^4$ &
					$\a \neq \b$, $\a\b \neq 0$
					\\ \hline
					
					\end{longtable}
				}}
			\end{center}
\end{proposition}

\begin{proof}
	(1) (Type S$'$)\,\,
	We prove by the similar way of Type T$_{1}'$ as in (2-1). 
%
%
%

\smallskip

\noindent
(2-1) (Type T$_{1}'$)\,\,
Let $A$ be a geometric algebra of Type T$_{1}'$.
		By Theorem \ref{thm.DR}, the defining relations of $A$ are
		
			$\begin{cases}
				g_1=x^2y-\a^2yx^2+\b yxy-\a\b y^2x, \\
				g_2=xy^2-\a^2y^2x, 
			\end{cases}
			(\b \in k,\,\a \in k\setminus \{0\})
			$. 
		If $A$ is a $3$-dimensional cubic AS-regular algebra, then there exists a twisted superpotential $\om \in \kxy_4$
		such that $A=\cD(\om)$. In this case, $\om$ can be written as
		$\om=axg_1+bxg_2+cyg_1+dyg_2  \,(a,b,c,d \in k)$. 
		Since
			$$\begin{cases}
				\om \partial_x=-a\a^2xyx-a\a\b xy^2-b\a^2xy^2-c\a^2y^2x-c\a\b y^3-d\a^2y^3, \\
				\om \partial_y=ax^3+a\b xyx+b x^2y+c yx^2+c\b y^2x+d yxy,
			\end{cases}$$
		it follows from Lemma \ref{lem.TSP} that $a=0$ and $c\b+d\a^2=0$.
		Moreover, $c=-b\a^2$ and $d=b\b$,
		so we have \vspace{-5pt}
			$$\om=bx^2y^2-b\a^2yx^2y-b\a^2xy^2x+b\a^4y^2x^2+b\a\b yxy^2-b\b\a^2 y^2xy
			\quad (b \in k \setminus \{0\})
			$$
		By Lemma \ref{lem.TSP}, we also have $\a=1$, that is,
			$$
			\om=x^2y^2-yx^2y-xy^2x+y^2x^2+\b yxy^2-\b y^2xy.
					\vspace{-5pt}
			$$
		In this case, $\bM(\om)=
			\begin{pmatrix}
				-y^2 & xy \\
				yx & -x^2+\b xy-\b yx
			\end{pmatrix}
			$ and $\det(\bM(\om))=-\b(\yy)(\XY-\yx)$.
			Therefore, $A$ is of Type P if and only if $\b=0$.
			So, we may assume that $\b \neq 0$.
			In this case, $\partial_x \om, \partial_y \om$ are linearly independent and the common zero locus of entries of $\bM(\om)$ in $\PP^1 \times \PP^1$
			is equal to empty. 
			Therefore,  $A=\cD(\om)$ is AS-regular.
		

\smallskip

\noindent
(2-2) (Type T$_{2}'$)\,\,
Let $A$ be a geometric algebra of Type T$_{2}'$. 
By Theorem \ref{thm.DR}, the defining relations of $A$ are

$\begin{cases}
g_1=x^2y-yx^2+\a yxy+(2-\a)y^2x+(\a-2)y^3, \\
g_2=xy^2-y^2x+2y^3, 
\end{cases}
\,(\a \in k)
$. 
If $A$ is a $3$-dimensional cubic AS-regular algebra, then there exists a twisted superpotential $\om \in \kxy_4$
		such that $A=\cD(\om)$. In this case, $\om$ can be written as
		$\om=axg_1+bxg_2+cyg_1+dyg_2
		\quad (a,b,c,d \in k)
		$. 
		Since
		$$\begin{cases}
				\om \partial_x=-axyx+a(2-\a)xy^2-bxy^2-cy^2x-(c(\a-2)+d)y^3, \\
				\om \partial_y=ax^3+a\a xyx+bx^2y+cyx^2+(a(\a-2)+2b)xy^2+c\a y^2x \\
				\hfill
				+dyxy+(c(\a-2)+2d)y^3,
			\end{cases}$$

\noindent
		it follows from Lemma \ref{lem.TSP} that $a=0$, $c=-b$ and $d=\a b$,
		so we have 
			$$\om=x^2y^2-yx^2y-xy^2x+y^2x^2+2xy^3+\a yxy^2-\a y^2xy-2y^3x+(\a+2)y^4.$$
		Then $\bM(\om)=
		\begin{pmatrix}
			-y^2 & xy+2y^2 \\
			yx-2y^2 & -x^2+\a xy-\a yx+(\a+2)y^2
		\end{pmatrix}
		$ and 
		$$
		\det(\bM(\om))=(2-\a)(\yy)(\XY-\yx+\yy).
		$$
		Therefore, $A$ is of Type P if and only if $\a=2$. 
		So we may assume that $\a \neq 2$.
		In this case, $\partial_x \om, \partial_y \om$ are linearly independent and the common zero locus of entries of $\bM(\om)$ in $\PP^1 \times \PP^1$
		is equal to empty, so $A=\cD(\om)$ is AS-regular.
						
\smallskip

\noindent
(3-1) (Type FL$_{1}$)\,\,
		Let $A$ be a geometric algebra of Type FL$_{1}$. 
		By Theorem \ref{thm.DR}, the defining relations of $A$ are 
		$\begin{cases}
			g_1=x^2y-\a yx^2, \\
			g_2=xy^2-\b y^2x,
		\end{cases}
		\, (\a,\b \in k,\,\a\b \neq 0)
		$. 
	If $A$ is a $3$-dimensional cubic AS-regular algebra, then there exists a twisted superpotential $\om \in \kxy_4$
	such that $A=\cD(\om)$. In this case, $\om$ can be written as
		$\om=axg_1+bxg_2+cyg_1+dyg_2
		\quad (a,b,c,d \in k)
		$. 
	Since
		$\om \partial_x=-a\a xyx-b\b xy^2-c \a y^2x-d\b y^3$, 
		$\om \partial_y=ax^3+bx^2y+cy^2x+dyxy$, 
	it follows from Lemma \ref{lem.TSP} that $a=d=0$, $c=-b\a$ and $\a^2=\b^2$, 
	so we may assume that
		$\om=bx^2y^2-b\b xy^2x-b\a yx^2y+b\a^2 y^2x^2 \quad (b \in k\setminus \{0\})
		$.
	Then 
		$$\om=
		\begin{cases}
			x^2y^2-\a xy^2x-\a yx^2y+\a^2 y^2x^2 &\textnormal{ if } \b=\a, \\
			x^2y^2+\a xy^2x-\a yx^2y+\a^2 y^2x^2 &\textnormal{ if } \b=-\a.
		\end{cases}
		$$

		If $\b=\a$, then $\bM(\om)=
		\begin{pmatrix}
			\partial_x \om \partial_x & \partial_x \om \partial_y \\
			\partial_y \om \partial_x & \partial_y \om \partial_y
		\end{pmatrix}=
		\begin{pmatrix}
			-\a y^2 & xy \\
			\a^2 yx & -\a x^2
		\end{pmatrix}
		$ and $\det(\bM(\om))=0$. 
		This means that $A$ is of Type P.
		
		If $\b=-\a$, then $\bM(\om)=
		\begin{pmatrix}
			\a y^2 & xy \\
			\a^2 yx & -\a x^2
		\end{pmatrix}
		$ and $\det(\bM(\om))=-2\a^2(\xx)(\yy)$. 
		Therefore, $A$ is of Type P if and only if $\a=0$.
		So we may assume that $\a \neq 0$.
		In this case, $\partial_x \om, \partial_y \om$ are linearly independent and the common zero locus of entries of $\bM(\om)$ in $\PP^1 \times \PP^1$
		is equal to empty, so $A=\cD(\om)$ is AS-regular.

\smallskip

\noindent
(3-2) (Type FL$_{2}$)\,\,
Let $A$ be a geometric algebra of Type FL$_{2}$. 
By Theorem \ref{thm.DR}, the defining relations of $A$ are
		$\begin{cases}
			g_1=yxy-\a x^3, \\
			g_2=\b xyx-y^3,
		\end{cases}
		(\a,\b \in k,\a\b \neq 0)
		$. 
	If $A$ is a $3$-dimensional cubic AS-regular algebra, then there exists a twisted superpotential $\om \in \kxy_4$
	such that $A=\cD(\om)$. In this case, $\om$ can be written as
		$\om=axg_1+bxg_2+cyg_1+dyg_2
		\quad (a,b,c,d \in k)
		$.
	Since
		$\om \partial_x=-a\a x^3+b\b x^2y-c\a yx^2+d\b yxy$, 
		$\om \partial_y=axyx-bxy^2+cy^2x-dy^3$,
	it follows from Lemma \ref{lem.TSP} that $b=c=0$ and $a=d\b$, so
		$\om=-\a\b x^4+\b xyxy+\b yxyx-y^4$. 
	Then $\bM(\om)=
	\begin{pmatrix}
		-\a\b x^2 & \b yx \\
		\b xy & - y^2
	\end{pmatrix}
	$ and $\det(\bM(\om))=\b(\a-\b)(\xx)(\yy)$. 
	Hence, $A$ is of Type P if and only if $\a = \b$.
	So we may assume that $\a \neq \b$.
	In this case, $\partial_x \om, \partial_y \om$ are linearly independent and the common zero locus of entries of $\bM(\om)$ in $\PP^1 \times \PP^1$
	is equal to empty. 
	Therefore, $A=\cD(\om)$ is AS-regular.
\end{proof}
\section{Classifications of $3$-dimensional cubic AS-regular algebras of Tpe S$'$, T$'$ and FL. }\label{sec.Cond}
%

\subsection{{\bf Step 5:} Classify them up to isomorphism of graded algebras. }
\label{subsec.ISOM}
In this subsection, we will give the complete list of defining relations of $3$-dimensional cubic AS-regular algebras whose point schemes are not integral,
and classify them up to graded algebra isomorphism. 

%
%
%

Remark that Lemma \ref{lem.step4-1} plays an important role to classify $3$-dimensional cubic AS-regular algebras up to isomorphisms.

\begin{lemma}\label{lem.step4-1}
	Let $P=(1,0), Q=(0,1) \in \PP^1$, $\rho=
	\begin{pmatrix}
		\a & \b \\
		\g & \d 
	\end{pmatrix}
	\in \Aut_k \PP^1$ and $\tau_{1,1}=
	\begin{pmatrix}
		1 & 1 \\
		0 & 1
	\end{pmatrix} \in \Aut_k \PP^1$.
	\vspace{-10pt}
	\begin{enumerate}[{\rm (1)}]
		\item Let $E=\PP^1 \times \{P\} \cup \{P\} \times \PP^1 \cup C_{\id}$.
		If $(\rho \times \rho)(E)=E$, then $\rho=
		\begin{pmatrix}
			1 & \b \\
			0 & \d 
		\end{pmatrix}$ for $\b \in k$ and $\d \in k \setminus \{0\}$.
		\item Let $E=\PP^1 \times \{P\} \cup \{P\} \times \PP^1 \cup C_{\tau_{1,1}}$.
		If $(\rho \times \rho)(E)=E$, then $\rho=
		\begin{pmatrix}
			1 & \b \\
			0 & 1
		\end{pmatrix}$ for $\b \in k$.
		\item Let $E=\PP^1 \times \{P\} \cup \PP^1 \times \{Q\} \cup \{P\} \times \PP^1 \cup \{Q\} \times \PP^1$.
		If $(\rho \times \rho)(E)=E$, then $\rho=
		\begin{pmatrix}
			1 & 0 \\
			0 & \d 
		\end{pmatrix}$ or $\rho=
		\begin{pmatrix}
			0 & 1 \\
			\g & 0
		\end{pmatrix}$ for $\g\in k\setminus \{0\}\text{ and }\d \in k$.
	\end{enumerate}
\end{lemma}

%
%

%
\begin{proof}
%
%
		(1)\,\,Since $(\rho \times \rho)(E)=E$, $\rho(P)=P$ holds.
		Since $\rho(P)=(\a,\g)$, we have $\a \neq 0, \g=0$. 
		So $\rho=
		\begin{pmatrix}
			1 & \b \\
			0 & \d
		\end{pmatrix}$.	

\noindent			
		(2)\,\,Similarly to {\rm (1)}, we have $\rho=
		\begin{pmatrix}
			1 & \b \\
			0 & \d
		\end{pmatrix}$ for $\b \in k$ and $0 \neq \d \in k$. Since $(\rho \times \rho)(C_{\tau_{1,1}})=C_{\tau_{1,1}}$,
		it follows that $(\rho(p),\rho\tau_{1,1}(p)) \in C_{\tau_{1,1}}$ for any $p \in \PP^1$.
		Therefore, we have $\rho\tau_{1,1}=\tau_{1,1}\rho$. Since
		\begin{align*}
			& \rho\tau_{1,1}=
			\begin{pmatrix}
				1 & \b \\
				0 & \d
			\end{pmatrix}
			\begin{pmatrix}
				1 & 1 \\
				0 & 1
			\end{pmatrix}=
			\begin{pmatrix}
				1 & 1+\b \\
				0 & \d
			\end{pmatrix}, 
			\tau_{1,1}\rho=
			\begin{pmatrix}
				1 & 1 \\
				0 & 1
			\end{pmatrix}
			\begin{pmatrix}
				1 & \b \\
				0 & \d
			\end{pmatrix}=
			\begin{pmatrix}
				1 & \b+\d \\
				0 & \d
			\end{pmatrix},
		\end{align*}
		we have $\d=1$, so $\rho=
		\begin{pmatrix}
			1 & \b \\
			0 & 1
		\end{pmatrix}$. 

\noindent
		(3)\,\,Since $(\rho \times \rho)(E)=E$, it follows that
         $
			\begin{cases}
				\rho(P)=P, \\
				\rho(Q)=Q,
			\end{cases} \textnormal{or}\quad 
			\begin{cases}
				\rho(P)=Q, \\
				\rho(Q)=P.
			\end{cases}$
		By calculating, we have $\rho(P)=(\a,\g)$ and $\rho(Q)=(\b,\d)$.
		If $\rho(P)=P$ and $\rho(Q)=Q$, then $\a,\d \neq 0, \b=\g=0$, so
		$\rho=
		\begin{pmatrix}
			1 & 0 \\
			0 & \d
		\end{pmatrix}$. On the other hand, if $\rho(P)=Q$ and $\rho(Q)=P$, then $\a=\d=0, \b,\g \neq 0$, so
		$\rho=
		\begin{pmatrix}
			0 & 1 \\
			\g & 0
		\end{pmatrix}$.
\end{proof}
\begin{theorem}\label{thm.isom}
	Let $A=\cA(E,\sigma)$ be a $3$-dimensional cubic AS-regular algebra
	of Type S\,$'$, T\,$'$ or FL. 
	For each type, {\rm Table $3$} describes
	\begin{enumerate}
	\item[\rm (I):] the defining relations of $A$, and
	\item[\rm (II):] the conditions to be isomorphic as graded algebras in terms of their defining relations.
	\end{enumerate}
	
	\noindent
	In {\rm Table $3$}, if $X \neq Y$ or $i \neq j$, then Type $X_i$ algebra is not isomorphic to any Type $Y_j$ algebra.
	Moreover, every algebra in {\rm Table $3$} is a $3$-dimensional cubic AS-regular algebra.
	\begin{center}
		\noindent{\renewcommand\arraystretch{1.5}
			{\small
				\begin{longtable}{|p{0.8cm}|p{6cm}|p{5.0cm}|}
					\multicolumn{3}{c}{{\rm Table 3: 
\begin{minipage}[t]{20em}List of defining relations and conditions to be graded algebra isomorphic\end{minipage}}}
					\\[15pt] \hline
					{\rm Type}
					& {\rm (I) Defining relations \quad $(\a, \b \in k)$}
					& {\rm (II) Conditions to be graded algebra isomorphic}
					\\ \hline\hline
					{\rm S$'$} &
					$\begin{cases}
						xy^2-y^2x, \\
						x^2y+yx^2-2y^3
					\end{cases}$ &
					\ \hfill \textnormal{---------------------} \hfill \rule{0pt}{10pt}
					\\ \hline
					{\rm T$'_1$} &
					$\begin{cases}
						xy^2-y^2x, \\
						x^2y-yx^2+yxy-xy^2
					\end{cases}$ &
					\ \hfill \textnormal{---------------------} \hfill \rule{0pt}{10pt}
					\\ \hline
					{\rm T$'_2$} &
					$\begin{cases}
						xy^2-y^2x+2y^3, \\
						x^2y-yx^2 \\
						-\a xy^2+\a yxy+2y^2x  -(\a+2)y^3
					\end{cases}$ &
					$\a'=\a$
					\\ \hline
					{\rm FL$_1$} &
					$\begin{cases}
						xy^2+\a y^2x, \\
						x^2y-\a yx^2
					\end{cases}$ &
					$\a'=\a, -\a^{-1}$
					\\ \hline
					{\rm FL$_2$} &
					$\begin{cases}
						-\a x^3+yxy, \\
						\b xyx-y^3
					\end{cases}$ &
					$(\a',\b')=(\a,\b)$ \textnormal{in} $\PP^1$
					\\ \hline
				\end{longtable}
			}
		}
	\end{center}
\end{theorem}
\begin{proof}
	Let $P=(1,0), Q=(0,1) \in \PP^1$. 
We prove only for Type FL$_{1}$ and FL$_{2}$. 
For Type S$'$, T$_{1}'$ and T$_{2}'$, the proof is a similar way.

		\smallskip
		
		\noindent
		(3-1) (Type {\rm FL$_1$})\,\,
		Let $A$ be a $3$-dimensional cubic AS-regular algebra of Type {\rm FL$_1$}.
		By Theorem \ref{thm.DR} and Proposition \ref{prop.P}, we can write 
		\vspace{-10pt}
		\begin{align*}
			&A_{\a}:=A=\cA(E,\sigma_{\a})=\kxy/(xy^2+\a y^2x,x^2y-\a yx^2)
			\quad (\a \neq 0),\\
			&\text{where }E=\PP^1 \times \{P\} \cup \PP^1 \times \{Q\} \cup \{P\} \times \PP^1 \cup \{Q\} \times \PP^1, 
			\\ & 
			\begin{cases}
				\sigma_{\a}(P,p)=(p,P), \,\,
				\sigma_{\a}(Q,p)=(p,Q), \\
				\sigma_{\a}(p,P)=(P,\tau_{\a}(p)), \,\,
				\sigma_{\a}(p,Q)=(Q,\tau_{-\a}(p)).
			\end{cases}
			\vspace{-10pt}
		\end{align*}
		Assume that $A_{\a} \cong A_{\a'}$. 
		By Lemma \ref{thm.GA} (1), 
		there exists $\rho \in \Aut_k \PP^1$ such that
		$\rho \times \rho$ restricts to an automorphism of $E$ and
		{\small 
		$
		\xymatrix{
			E \ar[r]^{\rho \times \rho} \ar[d]_{\sigma_{\a}} & E \ar[d]^{\sigma_{\a'}} \\
			E \ar[r]_{\rho \times \rho} & E
		}
		$
		}
		commutes. 
		By Lemma \ref{lem.step4-1} (3), it holds that
		$\rho=\begin{pmatrix}
			1 & 0 \\
			0 & d
		\end{pmatrix}$ or $\rho=\begin{pmatrix}
			0 & 1 \\
			c & 0
		\end{pmatrix}$.
		If $\rho=\begin{pmatrix}
			1 & 0 \\
			0 & d
		\end{pmatrix}$, then 
			$\sigma_{\a'} \circ (\rho \times \rho)=(\rho \times \rho) \circ \sigma_{\a} \Longleftrightarrow \tau_{\a'}\rho=\rho\tau_{\a}$.
		Since $\rho\tau_{\a}\rho^{-1}=\tau_{\a}$, 
			$\tau_{\a'}\rho=\rho\tau_{\a} \Longleftrightarrow \a'=\a$.
		If $\rho=\begin{pmatrix}
			0 & 1 \\
			c & 0
		\end{pmatrix}$, then 
			$\sigma_{\a'} \circ (\rho \times \rho)=(\rho \times \rho) \circ \sigma_{\a} \Longleftrightarrow \tau_{\a'}\rho=\rho\tau\tau_{\a}$.
		Since $\rho\tau\tau_{\a}\rho^{-1}=
		\begin{pmatrix}
			1 & 0 \\
			0 & -\a^{-1}
		\end{pmatrix}$, 
			$\tau_{\a'}\rho=\rho\tau\tau_{\a} \Longleftrightarrow \a'=-\a^{-1}$.
		
		Conversely, if $\a'=\a$, then it is clear that $A_{\a'} \cong A_{\a}$
		as graded algebras.
		If $\a'=-\a^{-1}$, then we set $\rho:=
		\begin{pmatrix}
			0 & 1 \\
			1 & 0
		\end{pmatrix}$. By the direct calculation, we have $(\rho \times \rho)(E)=E$ and $\sigma_{\a'} \circ (\rho \times \rho)=(\rho \times \rho) \circ \sigma_{\a}$.
		By Lemma \ref{thm.GA} (1), 
		$A_{\a'} \cong A_{\a}$ as graded algebras.
		
		\smallskip
		
		\noindent
		(3-2) (Type {\rm FL$_2$})\,\,
		Let $A$ be a $3$-dimensional cubic AS-regular algebra of Type {\rm FL$_2$}.
		By Theorem \ref{thm.DR} and Proposition \ref{prop.P}, we can write
		\begin{align*}
			&A_{\a,\b}:=A=\cA(E,\sigma_{\a,\b})=\kxy/(yxy-\a x^3,\b xyx-y^3)
			\quad (\a \neq \b), \\
			&\text{where }E=\PP^1 \times \{P\} \cup \PP^1 \times \{Q\} \cup \{P\} \times \PP^1 \cup \{Q\} \times \PP^1, 
			\\ & 
			\begin{cases}
				\sigma_{\a,\b}(P,p)=(p,Q), \,\,
				\sigma_{\a,\b}(Q,p)=(p,P), \\
				\sigma_{\a,\b}(p,P)=(P,\mu_{\a}(p)), \,\,
				\sigma_{\a,\b}(p,Q)=(Q,\mu_{\b}(p)). 
			\end{cases}
		\end{align*}
		Assume that $A_{\a',\b'} \cong A_{\a,\b}$ as graded algebras. 
		By Lemma \ref{thm.GA} (1), 
		there exists $\rho \in \Aut_k \PP^1$ such that
		$\rho \times \rho$ restricts to an automorphism of $E$ and
		{\small
		$
		\xymatrix{
			E \ar[r]^{\rho \times \rho} \ar[d]_{\sigma_{\a,\b}} & E \ar[d]^{\sigma_{\a',\b'}} \\
			E \ar[r]_{\rho \times \rho} & E
		}
		$
		}
		commutes.
		By Lemma \ref{lem.step4-1} (3), it holds that
		$\rho=\begin{pmatrix}
			1 & 0 \\
			0 & d
		\end{pmatrix}$ or $\rho=\begin{pmatrix}
			0 & 1 \\
			c & 0
		\end{pmatrix}$. If $\rho=\begin{pmatrix}
			1 & 0 \\
			0 & d
		\end{pmatrix}$, then
			$\sigma_{\a',\b'} \circ (\rho \times \rho)=(\rho \times \rho) \circ \sigma_{\a,\b} \Longleftrightarrow \tau_{\a'}\rho=\rho\tau_{\a}, \tau_{\b'}\rho=\rho\tau_{\b}$.
		Since $\rho\tau_{\a}\rho^{-1}=
		\begin{pmatrix}
			0 & 1 \\
			d^2\a & 0 
		\end{pmatrix}$ and $\rho\tau_{\b}\rho^{-1}=
		\begin{pmatrix}
			0 & 1 \\
			d^2\b & 0 
		\end{pmatrix}$,
		\vspace{-5pt}
			$$\tau_{\a'}\rho=\rho\tau_{\a}, \tau_{\b'}\rho=\rho\tau_{\b} \Longleftrightarrow (\a',\b')=(\a,\b) \text{ in } \PP^1.\vspace{-5pt}$$
		If $\rho=\begin{pmatrix}
			0 & 1 \\
			c & 0
		\end{pmatrix}$, then 
			$\sigma_{\a',\b'} \circ (\rho \times \rho)=(\rho \times \rho) \circ \sigma_{\a,\b} \Longleftrightarrow \tau_{\b'}\rho=\rho\tau_{\a}, \tau_{\a'}\rho=\rho\tau_{\b}$.
		Since $\rho\tau_{\a}\rho^{-1}=
		\begin{pmatrix}
			0 & 1 \\
			\frac{c^2}{\a} & 0 
		\end{pmatrix}$,
			$\tau_{\b'}\rho=\rho\tau_{\a} \Longleftrightarrow \b'\a=c^2$.
		Similarly, it follows that
			$\tau_{\a'}\rho=\rho\tau_{\b} \Longleftrightarrow \a'\b=c^2$.
		Therefore, we have 
			$\sigma_{\a',\b'} \circ (\rho \times \rho)=(\rho \times \rho) \circ \sigma_{\a,\b} \Longleftrightarrow (\a',\b')=(\a,\b) \text{ in } \PP^1$.  
		
		Conversely, suppose that $(\a',\b')=(\a,\b)$ in $\PP^1$. 
Then there exists 
$\lambda \in k\setminus \{0\}$ such that
		$\a'=\lambda \a, \b'=\lambda \b$. We set $\rho:=
		\begin{pmatrix}
			1 & 0 \\
			0 & \sqrt{\lambda}
		\end{pmatrix}$.
		By the direct calculation, we have 
		$(\mu \times \mu)(E)=E \text{ and } \sigma_{\a',\b'} \circ (\rho \times \rho)=(\rho \times \rho) \circ \sigma_{\a,\b}$. 
		By Lemma \ref{thm.GA} (1), 
		$A_{\a',\b'} \cong A_{\a,\b}$ as graded algebras.
\end{proof}
\subsection{{\bf Step 6:} Classify them up to graded Morita equivalence in terms of their defining relations. }
\label{subsec.ME}
%
%
%
%

\begin{theorem}\label{thm.grmod}
	Let $A=\cA(E,\sigma)$ be a $3$-dimensional cubic AS-regular algebra
	of Type S\,$'$, T\,$'$ or FL. 
	For each type, {\rm Table $4$} describes 
	\begin{enumerate}
	\item[\rm (I):] the defining relations of $A$, and
	\item[\rm (III):] the conditions to be graded Morita equivalent in terms of their defining relations.	
	\end{enumerate}
	In {\rm Table $4$}, if $X \neq Y$, then Type $X$ algebra is not graded Morita equivalent to
	any Type $Y$ algebra.
	Moreover, every algebra in {\rm Table $4$} is a $3$-dimensional cubic AS-regular algebra.
	\begin{center}
		\noindent{\renewcommand\arraystretch{1.5}
			{\small
				\begin{longtable}{|p{0.8cm}|p{5.3cm}|p{5.0cm}|}
					\multicolumn{3}{c}{{\rm Table 4: 
\begin{minipage}[t]{20em}List of defining relations and conditions to be graded  Morita equivalent\end{minipage}}}
					\\[15pt] \hline
					{\rm Type} & {\rm (I) Defining relations \quad$(\a, \b \in k)$}
					& {\rm (III) Conditions to be graded Morita equivalent} 
					\\ \hline\hline
					{\rm S$'$} &
					$\begin{cases}
						xy^2-y^2x, \\
						x^2y+yx^2-2y^3
					\end{cases}$ &
					\ \hfill \textnormal{---------------------} \hfill \rule{0pt}{10pt}
					\\ \hline
					{\rm T$'$} &
					$\begin{cases}
						xy^2-y^2x, \\
						x^2y-yx^2+yxy-xy^2
					\end{cases}$ &
					\ \hfill \textnormal{---------------------} \hfill \rule{0pt}{10pt}
					\\ \hline
					{\rm FL} &
					$\begin{cases}
						-\a x^3+yxy, \\
						\b xyx-y^3
					\end{cases}$ &
					$(\a',\b')=(\a,\b), (\b,\a)$ \textnormal{in} $\PP^1$
					\\ \hline
				\end{longtable}
			}
		}
	\end{center}
\end{theorem}

\begin{proof}

We prove only for Type FL. 
For Type S$'$ and T$'$, the proof is a similar way. 

	   \smallskip
		
		\noindent
		(3-1) (Type {\rm FL$_1$})\,\,		
		Let $A$ be a $3$-dimensional cubic AS-regular algebra of Type {\rm FL$_1$}.
		By Theorem \ref{thm.DR} and Proposition \ref{prop.P}, we can write
		\begin{align*}
			&A_{\a}:=A=\cA(E,\sigma_{\a})=\kxy/(xy^2+\a y^2x,x^2y-\a yx^2) \quad (\a \neq 0),\\
			&\text{where }E=\PP^1 \times \{P\} \cup \PP^1 \times \{Q\} \cup \{P\} \times \PP^1 \cup \{Q\} \times \PP^1, 
			\\ & 
			\begin{cases}
				\sigma_{\a}(P,p)=(p,P), \,\,
				\sigma_{\a}(Q,p)=(p,Q), \\
				\sigma_{\a}(p,P)=(P,\tau_{\a}(p)), \,\,
				\sigma_{\a}(p,Q)=(Q,\tau_{-\a}(p)).
			\end{cases}
		\end{align*}
		We will show that every $A_{\a}$ is graded Morita equivalent to $A_{1}$.
		For every $n \in \ZZ$, we set
			$\rho_{2n}:=
			\begin{pmatrix}
				1 & 0 \\
				0 & \a^{-n}
			\end{pmatrix}$ and
			$\rho_{2n+1}:=
			\begin{pmatrix}
				1 & 0 \\
				0 & \a^{-n}
			\end{pmatrix}$.
		It is clear that $\rho_{i} \in \Aut_k \PP^1$ and $\rho_{i} \times \rho_{i+1}$ restricts to an automorphism of $E$ for every $i \in \ZZ$.
		For every $i \in \ZZ$,
		\vspace{-10pt}
		\begin{align*}
			&\begin{cases}
				(\sigma_{1} \circ (\rho_{i} \times \rho_{i+1}))(P,p)=(\rho_{i+1}(p),P), \\
				(\sigma_{1} \circ (\rho_{i} \times \rho_{i+1}))(Q,p)=(\rho_{i+1}(p),Q), \\
				(\sigma_{1} \circ (\rho_{i} \times \rho_{i+1}))(p,P)=(P,\rho_{i}(p)), \\
				(\sigma_{1} \circ (\rho_{i} \times \rho_{i+1}))(p,Q)=(Q,\tau_{-1}\rho_{i}(p)),
			\end{cases}
			\\ & \begin{cases}
				((\rho_{i+1} \times \rho_{i+2}) \circ \sigma_{\a})(P,p)=(\rho_{i+1}(p),P), \\
				((\rho_{i+1} \times \rho_{i+2}) \circ \sigma_{\a})(Q,p)=(\rho_{i+1}(p),Q), \\
				((\rho_{i+1} \times \rho_{i+2}) \circ \sigma_{\a})(p,P)=(P, \rho_{i+2}\tau_{\a}(p)), \\
				((\rho_{i+1} \times \rho_{i+2}) \circ \sigma_{\a})(p,Q)=(Q,\rho_{i+2}\tau_{-\a}(p)).
				\vspace{-10pt}
			\end{cases}
		\end{align*}
		If $i=2n$ ($n \in \ZZ$), then
			$\rho_{i+2}\tau_{\a}=
			\begin{pmatrix}
				1 & 0 \\
				0 & \a^{-(n+1)}
			\end{pmatrix}
			\begin{pmatrix}
				1 & 0 \\
				0 & \a
			\end{pmatrix}=
			\begin{pmatrix}
				1 & 0 \\
				0 & \a^{-n}
			\end{pmatrix}=\rho_{i}$.
		If $i=2n+1$ ($n \in \ZZ$), then
			$\rho_{i+2}\tau_{\a}=
			\begin{pmatrix}
				1 & 0 \\
				0 & \a^{-(n+1)}
			\end{pmatrix}
			\begin{pmatrix}
				1 & 0 \\
				0 & \a
			\end{pmatrix}=
			\begin{pmatrix}
				1 & 0 \\
				0 & \a^{-n}
			\end{pmatrix}=\rho_{i}$.
			
		Similarly, we have $\rho_{i+2}\tau_{-\a}=\tau_{-1}\rho_{i+2}$ for any $i \in \ZZ$.
		Therefore, it follows that
			$\sigma_{1} \circ (\rho_{i} \times \rho_{i+1})=(\rho_{i+1} \times \rho_{i+2}) \circ \sigma_{\a}$		for every $i \in \ZZ$. 
			By Lemma \ref{thm.GA} (2), 
			we have $\GrMod A_{\a} \cong \GrMod A_1$. 
		
				\smallskip
		
		\noindent
		(3-2) (Type {\rm FL$_2$})\,\,		
		Let $B$ be a $3$-dimensional cubic AS-regular algebra of Type {\rm FL$_2$}.
		By Theorem \ref{thm.DR} and Proposition \ref{prop.P}, we can write
		\begin{align*}
			&B_{\b,\g}:=B=\cA(E,\sigma_{\b,\g})=\kxy/(yxy-\b x^3,\g xyx-y^3) \quad (\b \neq \g),\\
			&E=\PP^1 \times \{P\} \cup \PP^1 \times \{Q\} \cup \{P\} \times \PP^1 \cup \{Q\} \times \PP^1, 
			\begin{cases}
				\sigma_{\b,\g}(P,p)=(p,Q), \\
				\sigma_{\b,\g}(Q,p)=(p,P), \\
				\sigma_{\b,\g}(p,P)=(P,\mu_{\b}(p)), \\
				\sigma_{\b,\g}(p,Q)=(Q,\mu_{\g}(p)).
			\end{cases}
		\end{align*}
		We will show that $\GrMod A_{1} \cong \GrMod B_{1,-1}$. 
		We define a sequence $\{\rho_i\}_{i \in \ZZ}$ of automorphisms of $\PP^1$; 
			$\rho_i:=
			\begin{cases}
				\id, &\textnormal{if } i \equiv 0,1 \,\, (\mod 8), \\
				\mu_{1}, &\textnormal{if } i \equiv 2,7 \,\, (\mod 8), \\
				\mu_{-1}, &\textnormal{if }  i \equiv 3,6  \,\,  (\mod 8), \\
				\tau_{-1}, &\textnormal{if }  i \equiv 4,5 \,\,  (\mod 8),  \\
			\end{cases}$
where $\mu_{\a}=
\begin{pmatrix}
0 & 1 \\
\a & 0
\end{pmatrix}
$
for $\a\in k\setminus \{0\}$. 
		By direct calculation, the diagram
		{\small 
		$
		\xymatrix@C=40pt{
			E \ar@<0.5ex>[r]^{\rho_{i} \times \rho_{i+1}} \ar[d]_{\sigma_{1}} & E \ar[d]^{\sigma_{1,-1}} \\
			E \ar@<-0.5ex>[r]_{\rho_{i+1} \times \rho_{i+2}} & E
		}
		$}
		commutes for every $i \in \ZZ$.
		By Theorem \ref{thm.GA} (2),
			we have $\GrMod A_{1} \cong \GrMod B_{1,-1}$.
			
		We will show that $\GrMod B_{\b,\g} \cong \GrMod B_{\b',\g'}$ if and only if
			$(\b',\g')=(\b,\g), (\g,\b)$ in $\PP^1$.
		Assume that $\GrMod B_{\b,\g} \cong \GrMod B_{\b',\g'}$. 
		By Lemma \ref{thm.GA} (2), 
		there exists a sequence $\{\rho_{i}\}_{i \in \ZZ}$ of automorphisms of $\PP^1$ such that
		$\rho_{i} \times \rho_{i+1}$ restricts to an automorphism of $E$ and
		{\small 
		$
		\xymatrix@C=40pt{
			E \ar@<0.5ex>[r]^{\rho_{i} \times \rho_{i+1}} \ar[d]_{\sigma_{\b,\g}} & E \ar[d]^{\sigma_{\b',\g'}} \\
			E \ar@<-0.5ex>[r]_{\rho_{i+1} \times \rho_{i+2}} & E
		}
		$}
		commutes for every $i \in \ZZ$.		
		If $\rho_{i}=
		\begin{pmatrix}
			1 & 0 \\
			0 & d_i
		\end{pmatrix}, \rho_{i+1}=
		\begin{pmatrix}
			1 & 0 \\
			0 & d_{i+1}
		\end{pmatrix}, \rho_{i+2}=
		\begin{pmatrix}
			1 & 0 \\
			0 & d_{i+2}
		\end{pmatrix}$, then 
		\vspace{-10pt}
		\begin{align*}
			&\begin{cases}
				(\sigma_{\b',\g'} \circ (\rho_{i} \times \rho_{i+1}))(P,p)=(\rho_{i+1}(p),Q), \\
				(\sigma_{\b',\g'} \circ (\rho_{i} \times \rho_{i+1}))(Q,p)=(\rho_{i+1}(p),P), \\
				(\sigma_{\b',\g'} \circ (\rho_{i} \times \rho_{i+1}))(p,P)=(P,\mu_{\b'}\rho_{i}(p)), \\
				(\sigma_{\b',\g'} \circ (\rho_{i} \times \rho_{i+1}))(p,Q)=(Q,\mu_{\g'}\rho_{i}(p)),
			\end{cases}
			\\ &
			\begin{cases}
				((\rho_{i+1} \times \rho_{i+2}) \circ \sigma_{\b,\g})(P,p)=(\rho_{i+1}(p),Q), \\
				((\rho_{i+1} \times \rho_{i+2}) \circ \sigma_{\b,\g})(Q,p)=(\rho_{i+1}(p),P), \\
				((\rho_{i+1} \times \rho_{i+2}) \circ \sigma_{\b,\g})(p,P)=(P, \rho_{i+2}\mu_{\b}(p)), \\
				((\rho_{i+1} \times \rho_{i+2}) \circ \sigma_{\b,\g})(p,Q)=(Q,\rho_{i+2}\mu_{\g}(p)).
			\end{cases}
		\end{align*}
		In this case, 
				\begin{align*}
			&\sigma_{\b',\g'} \circ (\rho_{i} \times \rho_{i+1})=(\rho_{i+1} \times \rho_{i+2}) \circ \sigma_{\b,\g} 
			\Longleftrightarrow 
			\,\,\,\mu_{\b'}\rho_{i}=\rho_{i+2}\mu_{\b}, \quad \mu_{\g'}\rho_{i}=\rho_{i+2}\mu_{\g} \\
			&\Longrightarrow \,\,\, \frac{\b'}{\b}=\frac{\g'}{\g} 
			\Longleftrightarrow \,\,\, (\b',\g')=(\b,\g) \,\,\, \textnormal{in}\,\,\, \PP^1.
		\end{align*}
		If $\rho_{i}=
		\begin{pmatrix}
			1 & 0 \\
			0 & d_i
		\end{pmatrix}, \rho_{i+1}=
		\begin{pmatrix}
			0 & 1 \\
			c_{i+1} & 0
		\end{pmatrix} \text{ and }\rho_{i+2}=
		\begin{pmatrix}
			1 & 0 \\
			0 & d_{i+2}
		\end{pmatrix}$, then
		\begin{align*}
			&
\begin{cases}
				(\sigma_{\b',\g'} \circ (\rho_{i} \times \rho_{i+1}))(P,p)=(\rho_{i+1}(p),Q), \\
				(\sigma_{\b',\g'} \circ (\rho_{i} \times \rho_{i+1}))(Q,p)=(\rho_{i+1}(p),P), \\
				(\sigma_{\b',\g'} \circ (\rho_{i} \times \rho_{i+1}))(p,P)=(Q,\mu_{\g'}\rho_{i}(p)), \\
				(\sigma_{\b',\g'} \circ (\rho_{i} \times \rho_{i+1}))(p,Q)=(P,\mu_{\b'}\rho_{i}(p)),
			\end{cases}
			\\ &
			\begin{cases}
				((\rho_{i+1} \times \rho_{i+2}) \circ \sigma_{\b,\g})(P,p)=(\rho_{i+1}(p),Q), \\
				((\rho_{i+1} \times \rho_{i+2}) \circ \sigma_{\b,\g})(Q,p)=(\rho_{i+1}(p),P), \\
				((\rho_{i+1} \times \rho_{i+2}) \circ \sigma_{\b,\g})(p,P)=(Q, \rho_{i+2}\mu_{\b}(p)), \\
				((\rho_{i+1} \times \rho_{i+2}) \circ \sigma_{\b,\g})(p,Q)=(P,\rho_{i+2}\mu_{\g}(p)).
			\end{cases}
		\end{align*}
		In this case,
		\begin{align*}
			&\sigma_{\b',\g'} \circ (\rho_{i} \times \rho_{i+1})=(\rho_{i+1} \times \rho_{i+2}) \circ \sigma_{\b,\g}
			\Longleftrightarrow \,\,\,\mu_{\g'}\rho_{i}=\rho_{i+2}\mu_{\b}, \quad \mu_{\b'}\rho_{i}=\rho_{i+2}\mu_{\g} \\
			&\Longrightarrow \,\,\, \frac{\g'}{\b}=\frac{\b'}{\g}
			\Longleftrightarrow \,\,\, (\b',\g')=(\g,\b) \,\,\, \textnormal{in}\,\,\, \PP^1.
		\end{align*}
		
		If $\rho_{i}=
		\begin{pmatrix}
			0 & 1 \\
			c_i & 0
		\end{pmatrix}, \rho_{i+1}=
		\begin{pmatrix}
			1 & 0 \\
			0 & d_{i+1}
		\end{pmatrix}
		\text{ and }
		\rho_{i+2}=
		\begin{pmatrix}
			0 & 1 \\
			c_{i+2} & 0
		\end{pmatrix}$, then
		\begin{align*}
			&\begin{cases}
				(\sigma_{\b',\g'} \circ (\rho_{i} \times \rho_{i+1}))(P,p)=(\rho_{i+1}(p),P), \\
				(\sigma_{\b',\g'} \circ (\rho_{i} \times \rho_{i+1}))(Q,p)=(\rho_{i+1}(p),Q), \\
				(\sigma_{\b',\g'} \circ (\rho_{i} \times \rho_{i+1}))(p,P)=(P,\mu_{\b'}\rho_{i}(p)), \\
				(\sigma_{\b',\g'} \circ (\rho_{i} \times \rho_{i+1}))(p,Q)=(Q,\mu_{\g'}\rho_{i}(p)),
			\end{cases}
			\\ & 
			\begin{cases}
				((\rho_{i+1} \times \rho_{i+2}) \circ \sigma_{\b,\g})(P,p)=(\rho_{i+1}(p),P), \\
				((\rho_{i+1} \times \rho_{i+2}) \circ \sigma_{\b,\g})(Q,p)=(\rho_{i+1}(p),Q), \\
				((\rho_{i+1} \times \rho_{i+2}) \circ \sigma_{\b,\g})(p,P)=(P, \rho_{i+2}\mu_{\b}(p)), \\
				((\rho_{i+1} \times \rho_{i+2}) \circ \sigma_{\b,\g})(p,Q)=(Q,\rho_{i+2}\mu_{\g}(p)).
			\end{cases}
		\end{align*}
		In this case,
		\begin{align*}
			&\sigma_{\b',\g'} \circ (\rho_{i} \times \rho_{i+1})=(\rho_{i+1} \times \rho_{i+2}) \circ \sigma_{\b,\g}
			\Longleftrightarrow \,\,\,\mu_{\b'}\rho_{i}=\rho_{i+2}\mu_{\b}, \quad \mu_{\g'}\rho_{i}=\rho_{i+2}\mu_{\g} \\
			&\Longrightarrow \,\,\, \b'\b=\g'\g 
			\Longleftrightarrow \,\,\, (\b',\g')=(\g,\b) \,\,\, \textnormal{in}\,\,\, \PP^1.
		\end{align*}		
		If $\rho_{i}=
		\begin{pmatrix}
			0 & 1 \\
			c_i & 0
		\end{pmatrix}, \rho_{i+1}=
		\begin{pmatrix}
			0 & 1 \\
			c_{i+1} & 0
		\end{pmatrix}, \rho_{i+2}=
		\begin{pmatrix}
			0 & 1 \\
			c_{i+2} & 0
		\end{pmatrix}$, then 
		\vspace{-5pt}
		\begin{align*}
			&\begin{cases}
				(\sigma_{\b',\g'} \circ (\rho_{i} \times \rho_{i+1}))(P,p)=(\rho_{i+1}(p),P), \\
				(\sigma_{\b',\g'} \circ (\rho_{i} \times \rho_{i+1}))(Q,p)=(\rho_{i+1}(p),P), \\
				(\sigma_{\b',\g'} \circ (\rho_{i} \times \rho_{i+1}))(p,P)=(Q,\mu_{\g'}\rho_{i}(p)), \\
				(\sigma_{\b',\g'} \circ (\rho_{i} \times \rho_{i+1}))(p,Q)=(P,\mu_{\b'}\rho_{i}(p)),
			\end{cases}
			\\ & 
			\begin{cases}
				((\rho_{i+1} \times \rho_{i+2}) \circ \sigma_{\b,\g})(P,p)=(\rho_{i+1}(p),P), \\
				((\rho_{i+1} \times \rho_{i+2}) \circ \sigma_{\b,\g})(Q,p)=(\rho_{i+1}(p),Q), \\
				((\rho_{i+1} \times \rho_{i+2}) \circ \sigma_{\b,\g})(p,P)=(Q, \rho_{i+2}\mu_{\b}(p)), \\
				((\rho_{i+1} \times \rho_{i+2}) \circ \sigma_{\b,\g})(p,Q)=(P,\rho_{i+2}\mu_{\g}(p)).
			\end{cases}
			\vspace{-10pt}
		\end{align*}
		In this case,\vspace{-10pt}
		\begin{align*}
			&\sigma_{\b',\g'} \circ (\rho_{i} \times \rho_{i+1})=(\rho_{i+1} \times \rho_{i+2}) \circ \sigma_{\b,\g}
			\Longleftrightarrow \,\,\,\mu_{\g'}\rho_{i}=\rho_{i+2}\mu_{\b}, \quad \mu_{\b'}\rho_{i}=\rho_{i+2}\mu_{\g} \\
			&\Longrightarrow \,\,\, \g'\b=\b'\g 
			\Longleftrightarrow \,\,\, (\b',\g')=(\b,\g) \,\,\, \textnormal{in}\,\,\, \PP^1.
			\vspace{-10pt}
		\end{align*}
		
		Conversely, if $(\b',\g')=(\b,\g)$ in $\PP^1$, then it is clear that $B_{\b',\g'}$ is graded Morita equivalent to $\GrMod B_{\b,\g}$
		because $B_{\b',\g'}$ is isomorphic to $B_{\b,\g}$ by Theorem \ref{thm.isom},
		so suppose that 
			$(\b',\g')=(\g,\b)$ in $\PP^1$.
		Then we define a sequence $\{\rho_i\}_{i \in \ZZ}$; 
			$\rho_i:=
			\begin{cases}
				\begin{pmatrix}
					0 & 1 \\
					1 & 0
				\end{pmatrix} 
				&\textnormal{ if }i \equiv 0 \,\,\, (\mod 4), \\
				\begin{pmatrix}
					0 & 1 \\
					\b'\b & 0
				\end{pmatrix}
				&\textnormal{ if } i \equiv 2 \,\,\, (\mod 4), \\
				\id &\textnormal{ if } i \equiv 1,3  \,\,\, (\mod 4).
			\end{cases}$
		By direct calculation, it holds that
		$\rho_{i} \times \rho_{i+1}$ restricts to an automorphism of $E$ and
		{\small 
		$
		\xymatrix@C=40pt{
			E \ar@<0.5ex>[r]^{\rho_{i} \times \rho_{i+1}} \ar[d]_{\sigma_{\b,\g}} & E \ar[d]^{\sigma_{\b',\g'}} \\
			E \ar@<-0.5ex>[r]_{\rho_{i+1} \times \rho_{i+2}} & E
		}
		$
		}
		commutes for every $i \in \ZZ$. 
		By Lemma \ref{thm.GA} (2),
			we have $\GrMod B_{\b',\g'} \cong \GrMod B_{\b,\g}$.
\end{proof}
\subsection{Summary}\label{subsec.SUM}
In conclusion, 
we give the complete list of defining relations of $3$-dimensional cubic AS-regular algebras whose point schemes are not integral. 
Moreover, we classify them up to isomorphisms of graded algebras and graded Morita equivalence in terms of their defining relations. 
Finally, 
we summarize the results in the tables as follows: 
\begin{center}
		\noindent{\renewcommand\arraystretch{1.5}
			{\small
				\begin{longtable}{|p{0.8cm}|p{6cm}|p{5.1cm}|}
					\multicolumn{3}{c}{{\rm Table: ISOM}}
					\\ \hline
					{\rm Type}
					& {\rm (I) Defining relations \quad$(\a, \b \in k)$}
					& {\rm (II) Conditions to be graded algebra isomorphic}
					\\ \hline\hline
					{\rm P$_{1}$} &
					$\begin{cases}
						x^{2}y-\a yx^{2}, \\
						xy^{2}-\a y^{2}x \quad (\a\neq 0)
					\end{cases}$ &
					$\a'=\a^{\pm 1}$
					\\ \hline
					{\rm P$_{2}$} &
					$\begin{cases}
						x^{2}y-yx^{2}+yxy, \\
						xy^{2}-y^{2}x+y^{3}
					\end{cases}$
					&
					\ \hfill \textnormal{---------------------} \hfill \rule{0pt}{10pt}
					\\ \hline
					{\rm S$_{1}$} &
					$\begin{cases}
						\a\b x^{2}y+(\a+\b)xyx+yx^{2},\\
						\a\b xy^{2}+(\a+\b)yxy+y^{2}x
					\end{cases}$
					
					$(\a\b\neq 0,\,{\a}^{2}\neq {\b}^{2})$
					&
                    $\{\a',\b'\}$ $=\{\a,\b \},\,\{\a^{-1},\b^{-1}\}$
					\\ \hline
					{\rm S$_{2}$} &
					$\begin{cases}
						xy^{2}+y^{2}x+(\a+\b)x^{3}, \\
						x^{2}y+yx^{2}+(\a^{-1}+\b^{-1})y^{3}
						\end{cases}$ 
						
						$(\a\b\neq0,\,{\a}^{2}\neq {\b}^{2})$
						&
                        $\dfrac{\a'}{\b'}=\left(\dfrac{\a}{\b}\right)^{\pm}$
					\\ \hline
					{\rm T$_{1}$} &
					$\begin{cases}
						x^{2}y-2xyx+yx^{2}-2(2\b-1)yxy\\
						\quad +2(2\b-1)xy^{2}+2\b(\b-1)y^{3},\\
						xy^{2}-2yxy+y^{2}x
					\end{cases}$ &
                    $\b'=\b,\,-\b$
					\\ \hline
					{\rm T$_{2}$} &
					$\begin{cases}
						x^{2}y+2xyx+yx^{2}+2y^{3}, \\
						xy^{2}+2yxy+y^{2}x
					\end{cases}$ &
					\ \hfill \textnormal{---------------------} \hfill \rule{0pt}{10pt}
					\\ \hline
					{\rm S$'$} &
					$\begin{cases}
						xy^2-y^2x, \\
						x^2y+yx^2-2y^3
					\end{cases}$ &
					\ \hfill \textnormal{---------------------} \hfill \rule{0pt}{10pt}
					\\ \hline
					{\rm T$'_1$} &
					$\begin{cases}
						xy^2-y^2x, \\
						x^2y-yx^2+yxy-xy^2
					\end{cases}$ &
					\ \hfill \textnormal{---------------------} \hfill \rule{0pt}{10pt}
					\\ \hline
					{\rm T$'_2$} &
					$\begin{cases}
						xy^2-y^2x+2y^3, \\
						x^2y-yx^2 \\
						-\a xy^2+\a yxy+2y^2x-(\a+2)y^3
					\end{cases}$ &
					$\a'=\a$
					\\ \hline
					{\rm FL$_1$} &
					$\begin{cases}
						xy^2+\a y^2x, \\
						x^2y-\a yx^2
					\end{cases}$ &
					$\a'=\a, -\a^{-1}$
					\\ \hline
					{\rm FL$_2$} &
					$\begin{cases}
						-\a x^3+yxy, \\
						\b xyx-y^3
					\end{cases}$ &
					$(\a',\b')=(\a,\b)$ \textnormal{in} $\PP^1$
					\\ \hline
					{\rm TWL} &
					$\begin{cases}
						xy^2+y^2x, \\
						x^2y+yx^2+y^3
					\end{cases}$ &
					\ \hfill \textnormal{---------------------} \hfill \rule{0pt}{10pt}
					\\ \hline
					{\rm WL$_1$} &
					$\begin{cases}
						\a^2xy^2+y^2x-2\a yxy, \\
						\a^2x^2y+yx^2-2\a xyx
					\end{cases}$ &
					$\a'=\a^{\pm 1}$
					\\ \hline
					{\rm WL$_2$} &
					$\begin{cases}
						xy^2+y^2x-2yxy, \\
						x^2y+yx^2-2xyx+4xy^2-4yxy
						\\ \hfill +2y^3
					\end{cases}$ &
					\ \hfill \textnormal{---------------------} \hfill \rule{0pt}{10pt}
					\\ \hline
				\end{longtable}
			}
		}
	\end{center}
%
	\begin{center}
		\noindent{\renewcommand\arraystretch{1.5}
			{\small
				\begin{longtable}{|p{0.8cm}|p{6cm}|p{5.1cm}|}
					\multicolumn{3}{c}{{\rm Table: GME}}
					\\ \hline
					{\rm Type} & {\rm (I) Defining relations \quad$(\a, \b \in k)$}
					& {\rm (III) Conditions to be graded Morita equivalent} 
					\\ \hline\hline
					{\rm P} &
					$\begin{cases}
						x^{2}y-yx^{2}, \\
						xy^{2}-y^{2}x
					\end{cases}$ &
					\ \hfill \textnormal{---------------------} \hfill \rule{0pt}{10pt}
					\\ \hline
					{\rm S} &
					$\begin{cases}
						\a\b x^{2}y+(\a+\b)xyx+yx^{2},\\
						\a\b xy^{2}+(\a+\b)yxy+y^{2}x
					\end{cases}$
					
					$(\a\b\neq 0,\,{\a}^{2}\neq {\b}^{2})$
					&
                    $\dfrac{\a'}{\b'}=\left(\dfrac{\a}{\b}\right)^{\pm}$
					\\ \hline
					{\rm T} &
					$\begin{cases}
						x^{2}y-2xyx+yx^{2}-2yxy+2xy^{2},\\
						xy^{2}-2yxy+y^{2}x
					\end{cases}$ &
					\ \hfill \textnormal{---------------------} \hfill \rule{0pt}{10pt}
					\\ \hline
					{\rm S$'$} &
					$\begin{cases}
						xy^2-y^2x, \\
						x^2y+yx^2-2y^3
					\end{cases}$ &
					\ \hfill \textnormal{---------------------} \hfill \rule{0pt}{10pt}
					\\ \hline
					{\rm T$'$} &
					$\begin{cases}
						xy^2-y^2x, \\
						x^2y-yx^2+yxy-xy^2
					\end{cases}$ &
					\ \hfill \textnormal{---------------------} \hfill \rule{0pt}{10pt}
					\\ \hline
					{\rm FL} &
					$\begin{cases}
						-\a x^3+yxy, \\
						\b xyx-y^3
					\end{cases}$ &
					$(\a',\b')=(\a,\b), (\b,\a)$ \textnormal{in} $\PP^1$
					\\ \hline
					{\rm TWL} &
					$\begin{cases}
						xy^2+y^2x, \\
						x^2y+yx^2+y^3
					\end{cases}$ &
					\ \hfill \textnormal{---------------------} \hfill \rule{0pt}{10pt}
					\\ \hline
					{\rm WL} &
					$\begin{cases}
						xy^2+y^2x-2yxy, \\
						x^2y+yx^2-2xyx
					\end{cases}$ &
					\ \hfill \textnormal{---------------------} \hfill \rule{0pt}{10pt}
					\\ \hline
				\end{longtable}
			}
		}
	\end{center}	

\section*{Acknowledgments}
We are grateful to the third author's Ph.D. supervisor, Professor Izuru Mori,  
for his support and helpful comments.
The first author was supported by JSPS Grant-in-Aid for Scientific Research (C) 24K06653. 

\end{document}